\documentclass[preprint,12pt]{elsarticle}
\biboptions{numbers,sort&compress}
\usepackage[T1]{fontenc}
\usepackage{lmodern}
\usepackage{amsmath,amssymb,amsthm,mathtools}
\usepackage{graphicx}
\usepackage{float}
\usepackage{microtype}
\usepackage{needspace}
\usepackage{hyperref}
\hypersetup{hidelinks,pdftitle={A calculus of scalar maximally monotone operators for aggregation and tail control},pdfauthor={Cosme Louart}}
\DeclareMathOperator{\gra}{gra}
\DeclareMathOperator{\dom}{dom}
\DeclareMathOperator{\ran}{ran}
\DeclareMathOperator{\id}{id}
\DeclareMathAlphabet{\mathbbm}{U}{fplmbb}{m}{n}
\newcommand{\R}{\mathbb R}
\newcommand{\E}{\mathbb E}
\newcommand{\PP}{\mathbb P}
\newcommand{\Mu}{\mathcal M_{\uparrow}}
\newcommand{\Md}{\mathcal M_{\downarrow}}
\newcommand{\SSS}{\mathcal S}
\newcommand{\TT}{\mathcal T}
\newtheorem{theorem}{Theorem}[section]
\newtheorem{proposition}[theorem]{Proposition}
\newtheorem{lemma}[theorem]{Lemma}
\newtheorem{corollary}[theorem]{Corollary}
\theoremstyle{definition}
\newtheorem{definition}[theorem]{Definition}
\newtheorem{example}[theorem]{Example}
\theoremstyle{remark}
\newtheorem{remark}[theorem]{Remark}
\numberwithin{equation}{section}
\allowdisplaybreaks[1]
\journal{Journal of Mathematical Analysis and Applications}
\begin{document}
\begin{frontmatter}
\title{A calculus of scalar maximally monotone operators for aggregation and tail control}
\author{Cosme Louart}
\ead{cosmelouart@cuhk.edu.cn}
\affiliation{organization={The Chinese University of Hong Kong, Shenzhen},
             city={Shenzhen}, country={China}}
\begin{abstract}
We give simple criteria for sums, nonnegative products and compositions of scalar maximally monotone operators to preserve maximality, and introduce a pointwise resolvent order with associated infima and suprema.
We first apply this framework to the Hardy transform to derive sharp bounds for nonlinear aggregation in terms of marginal upper-tail means, without independence assumptions.
We then employ a geometric version of the Hardy transform (the geometric mean operator) to establish concentration inequalities for nonlinear observations with multiplicative random Lipschitz controls.
No exponential tail decay is required, and the result covers heavy-tailed settings under logarithmic integrability assumptions.
Feed-forward neural networks illustrate the latter result.
\end{abstract}
\end{frontmatter}
\section*{Introduction}
\addcontentsline{toc}{section}{Introduction}

This work builds on Rockafellar and Royset's use of maximally monotone operators
to study probability distributions \cite{RockafellarRoyset2014}.
For a real random variable $X$, we consider its completed survival
operator and its inverse tail operator,
\[
 \SSS[X](t)=[\PP(X>t),\PP(X\geq t)],
 \qquad \TT[X]=\SSS[X]^{-1}.
\]
Rockafellar and Royset use the corresponding distribution and increasing
quantile operators.\footnote{In our notation, their operators are
$1-\SSS[X]$ and $\TT[X]\circ(1-\id)$, respectively.}
Both $\SSS[X]$ and $\TT[X]$ are maximally nonincreasing operators:
their graphs admit no proper nonincreasing extension.
Completing a monotone graph fills each jump with the interval between
its one-sided limits and includes the required values at domain endpoints.
The resulting values are empty sets, singletons or closed intervals,
possibly unbounded. This representation handles atoms and gaps in the
distribution, and gives a unique inverse relation without choosing
a convention at discontinuities.
With this notation, classical concentration inequalities become upper
bounds on $\TT[X]$: for any maximally nonincreasing operator $q$,
\begin{equation}\label{eq:intro-tail-comparison}
 \TT[X]\leq q
 \quad\Longleftrightarrow\quad
 \forall(p,t)\in\gra q:\qquad \PP(X>t)\leq p.
\end{equation}
Here $\gra q=\{(p,t):t\in q(p)\}$, and the operator inequality uses
the pointwise resolvent order introduced in Definition~\ref{def:order};
the equivalence is proved in Lemma~\ref{lem:survivalcomparison}.

Our first objective is to develop this approach into a consistent
calculus for maximally monotone operators on $\mathbb R$, where the
total order permits useful simplifications of the general Hilbert space
theory presented for instance in \cite{BauschkeCombettes2017}.
Algebraic operations on completed graphs
need not preserve maximality: even a product of two nonnegative
maximal monotone operators may fail to be maximal. A composition can
also fail to be monotone when the inverse of the inner operator and
the outer operator are both multivalued at the same intermediate point.
Building on classical monotone-operator theory
\cite{Minty1962,BauschkeCombettes2017}, we give explicit new scalar criteria
for sums, products and compositions
(Propositions~\ref{prop:sums}, \ref{prop:products} and
\ref{prop:composition}). We introduce the \emph{pointwise resolvent
order}, based on pointwise comparison of resolvents
(Definition~\ref{def:order}), and use it to construct suprema and
infima (Theorem~\ref{thm:lattice}). On survival operators this order
agrees with first-order stochastic dominance \cite{Nendel2020}.
We also recall the definition of the Aumann integral \cite{Aumann1965},
completing the set of operations needed to compare and average tails
while retaining their completed graphs.

Section~\ref{sec:algebra} develops this scalar calculus through modest but necessary
extensions and refinements of the classical theory. Unless otherwise
specified, its results are, to our knowledge, new. In particular, the
definition of the pointwise resolvent order is new; the operator integral
revisits Aumann integration within the framework of maximally monotone
graphs, with endpoint values specified by maximal completion.
The main contributions are the applications in Sections~\ref{sec:hardy}
and~\ref{sec:local}, where this calculus yields sharp aggregation bounds
and concentration estimates for nonlinear observations with random
Lipschitz controls.

Our first application concerns expected shortfall \cite{AcerbiTasche2002},
also called the superquantile \cite{RockafellarRoyset2014}
or conditional value-at-risk (CVaR) \cite{RockafellarUryasev2002CVaR}.
It is obtained by applying the Hardy transform to the tail quantile
operator $\TT[X]$. More generally, for an integrable nonincreasing
operator $f$ on $(0,1)$, this transform is the operator integral of its dilations
(Definition~\ref{def:hardytransform}):
\[
 \forall p\in(0,1):\qquad
 \mathcal H(f)(p)=\int_0^1 f(rp)\,dr
                =\frac1p\int_0^p f(s)\,ds.
\]
For $0<p<1$, this integral is single-valued and can be computed
from any measurable selection of $f$.\footnote{A \emph{selection} of a
set-valued operator $f$ on $A\subseteq\dom f$ is a single-valued function
$s:A\to\R$ such that, for every $x\in A$, $s(x)\in f(x)$.
It is a \emph{measurable selection} if $s$ is measurable.}
For $f=\TT[X]$, it gives the mean of the upper $p$-tail, including
the appropriate fraction of an atom at the threshold.
This connects tail averaging to aggregation.
Let $X_1,\ldots,X_n$ be nonnegative integrable random variables,
and let $F:[0,\infty)^n\to[0,\infty)$ be continuous, coordinatewise
nondecreasing and positively homogeneous. Writing $F^\sharp$ for its
least superadditive majorant \cite{GrecoMesiarRindoneSipeky2016},
Proposition~\ref{prop:hardytailaggregation} gives
\[
 \mathcal H\bigl(\TT[F(X_1,\ldots,X_n)]\bigr)(p)
 \leq F^\sharp\bigl(\mathcal H(\TT[X_1])(p),\ldots,
                         \mathcal H(\TT[X_n])(p)\bigr).
\]
Theorem~\ref{thm:hardytailaggregationsharp} proves that this bound is
sharp when $p$ and the marginal upper-tail means are fixed, while
the marginal laws may vary. The corresponding operator inequality
\eqref{eq:hardyaggregation} applies to lattice medians, defined through
suprema and infima in Proposition~\ref{prop:latticemedian}.
For random variables, taking $F$ to be an odd median gives
Corollary~\ref{cor:hardymedian}; taking $F$ to be the sum gives
$F=F^\sharp$ and recovers the classical subadditivity of expected shortfall
\cite{AcerbiTasche2002,RockafellarUryasev2002CVaR}.
The sum bound uses marginal upper-tail means at the same probability
level $p$; a union bound
allocates that level among the summands.
For fixed integrable laws $\mu_1,\ldots,\mu_n$,
Proposition~\ref{prop:hardyoptimality} gives the lattice identity
\[
 \frac1n\sum_{i=1}^n\mathcal H(\TT[\mu_i])
 =\sup_{\substack{m\geq1\\X_{i,j}\sim\mu_i\\
                   1\leq i\leq n,\ 1\leq j\leq m}}
   \TT\!\left[\frac1{nm}\sum_{i=1}^n\sum_{j=1}^mX_{i,j}\right].
\]
The supremum ranges over all couplings and all numbers of copies $m$.
This extends the common-law result of Wang and Wang
\cite[Corollary~3.7]{WangWang2015END} to heterogeneous integrable laws,
without the tail-density assumptions in the earlier heterogeneous
result of Puccetti, Wang and Wang
\cite[Theorem~4.2]{PuccettiWangWang2013}.

Our second application uses geometric Hardy averaging to control
nonlinear observations with random Lipschitz bounds. For a positive
nonincreasing operator $f$ on $(0,1)$ with integrable logarithm, define
\[
 \mathcal G(f)(p)
 =\exp\bigl(\mathcal H(\log\circ f)(p)\bigr)
 =\exp\left(\frac1p\int_0^p\log f(s)\,ds\right).
\]
This is the geometric mean operator in the classical P\'olya--Knopp
inequality \cite{HardyLittlewoodPolya1952,Luor2008GeometricMean};
Definition~\ref{def:geometrichardytransform} also treats zeros and
values at domain endpoints. Let $Z$ take values in a metric space $(E,d_E)$,
and let $Z'$ be an independent copy. Let $\TT_\Delta$ be a maximally
nonincreasing operator with domain interior $(0,1)$ and nonnegative
values there, such that
\[
 \forall f:E\to\R\text{ $1$-Lipschitz}:\qquad
 \TT[|f(Z)-f(Z')|]\leq\TT_\Delta.
\]
Such uniform control of Lipschitz observables is a standard conclusion
of concentration of measure theory. For example, if
$Z\sim\mathcal N(0,I_d)$ in Euclidean space, Gaussian concentration
\cite[Proposition~5.34]{Vershynin2012Introduction} allows the choice
\[
 \forall u\in(0,1):\qquad \TT_\Delta(u)=2\sqrt{\log(2/u)},
\]
with maximal decreasing completion at the endpoints.\footnote{For this
choice, the quantile comparison is equivalent to requiring, for every
$t\geq0$ and every real $1$-Lipschitz function $f$,
$\PP(|f(Z)-f(Z')|\geq t)\leq2\exp(-t^2/4)$. The exponent follows by applying
Gaussian concentration to $(z,z')\mapsto f(z)-f(z')$, which has mean
zero and Lipschitz constant $\sqrt{2}$.}

Let a measurable map $\Phi:E\to E'$ satisfy the randomized Lipschitz control
\[
 \forall z,z'\in E:\qquad
 d_{E'}(\Phi(z),\Phi(z'))
 \leq d_E(z,z')\prod_{i=1}^n
             \max\{\Lambda_i(z),\Lambda_i(z')\},
\]
with nonnegative measurable controls $\Lambda_i$.
Under the logarithmic integrability assumptions of
Theorem~\ref{the:concentration_bounded_variations}, every real
$1$-Lipschitz function $g$ on $E'$ satisfies
\[
 \TT[|g(\Phi(Z))-g(\Phi(Z'))|]
 \leq\left.
 \left[\mathcal G\!\left(\TT_\Delta
                  \prod_{i=1}^n\TT[\Lambda_i(Z)]\right)
       \circ(\id/3)\right]\right|_{(0,1)}.
\]
Products are formed on $(0,1)$ and the restriction includes maximal
completion. 
The dilation is independent of the number of controls.
The P\'olya--Knopp inequality yields moment bounds, and
Proposition~\ref{prop:local-moment-sharpness} proves matching lower
bounds for prescribed input operators, up to constants depending
only on the moment order.
Finally, feed-forward neural networks provide an application of
the product control (Proposition~\ref{prop:local-network}):
we obtain explicit dependence on depth, compare the subgaussian
bounds with those of Adamczak \cite[Theorem~3.7]{Adamczak2025NonlinearMatrices},
and also treat power-tailed joint laws of the data and weights.

\tableofcontents

\section{Algebra with set-valued operators}\label{sec:algebra}

\subsection{Maximally monotone operators and their operations}\label{sec:operations}

\subsubsection{Scalar graphs and resolvents}
For a set-valued operator $f:\R\rightrightarrows\R$, write
\[
 \begin{aligned}
 \gra f&=\{(x,u):u\in f(x)\},\\
 \dom f&=\{x:f(x)\ne\varnothing\},\qquad
 \ran f=\bigcup_x f(x).
 \end{aligned}
\]
Its inverse is defined by $u\in f(x)\Longleftrightarrow x\in f^{-1}(u)$.
Functions are identified with singleton-valued operators. Infinite endpoints are permitted in interval notation, but all graph points belong to $\R^2$.

\begin{definition}
An operator is \emph{nondecreasing} if
\[
 \forall (x,u),(y,v)\in\gra f:\qquad (x-y)(u-v)\geq0,
\]
and \emph{nonincreasing} if $-f$ is nondecreasing. It is \emph{maximally monotone} in a specified orientation if its graph has no proper monotone extension of that orientation.
For a nondecreasing monotone operator, this is equivalent to
\begin{equation}\label{eq:maximality}
 \forall(x,u)\in\R^2,\quad
 \left[\forall(y,v)\in\gra f,\ (x-y)(u-v)\ge0\right]
 \Longrightarrow (x,u)\in\gra f.
\end{equation}
Equivalently, in contrapositive form,
\begin{equation}\label{eq:maximalitycontra}
 \forall(x,u)\notin\gra f,\quad
 \exists(y,v)\in\gra f:\ (x-y)(u-v)<0.
\end{equation}
For the nonincreasing orientation, replace $\ge0$ by $\le0$ in
\eqref{eq:maximality} and $<0$ by $>0$ in \eqref{eq:maximalitycontra}.
Denote these two classes by $\Mu$ and $\Md$.
\end{definition}
The orientation is always specified, including for a constant or a vertical graph, which belongs to both classes.

\begin{proposition}[Scalar Minty representation]\label{prop:minty}
For a monotone operator $f$, let $J_f$ denote its \emph{resolvent}, defined according to its orientation in the table below. The operator is maximal if and only if $J_f$ is an everywhere defined nondecreasing $1$-Lipschitz function. In that case, its graph is the following set, and the displayed map from $s\in\R$ onto $\gra(f)$ is a homeomorphism:
\begin{equation}\label{eq:minty}
 \begin{array}{c|cc}
  &J_f&\gra(f)\\ \hline
  f\text{ nondecreasing}&(\id+f)^{-1}
   &\{(J_f(s),s-J_f(s)):s\in\R\}\\[4pt]
  f\text{ nonincreasing}&(\id-f)^{-1}
   &\{(J_f(s),J_f(s)-s):s\in\R\}
 \end{array}
\end{equation}
Conversely, every everywhere defined nondecreasing $1$-Lipschitz function is a resolvent, in either orientation.
\end{proposition}
This is the scalar form of Minty's resolvent correspondence \cite{Minty1962}; see \cite{BauschkeCombettes2017}. On $\R$, firm nonexpansiveness is equivalent to being nondecreasing and $1$-Lipschitz; the decreasing orientation follows by applying the result to $-f$.

In particular, maximal graphs are closed and connected, the corresponding domains and ranges are intervals, and every nonempty operator value is a closed interval. Inversion preserves maximality and orientation.

\begin{proposition}[Pointwise characterization of maximality]\label{prop:pointwisemaximality}
Let $f:\R\rightrightarrows\R$ be a nondecreasing monotone operator,
possibly with empty values. Then $f$ is maximally monotone if and only if,
\begin{equation}\label{eq:maximalvalues}
\forall x\in\R:\qquad  f(x)=\left[\sup\bigcup_{y<x}f(y),\ \inf\bigcup_{y>x}f(y)\right]\cap\R,
\end{equation}
where the supremum and infimum are taken in the extended real line,
$\sup\varnothing=-\infty$, $\inf\varnothing=+\infty$.
\end{proposition}
The identity applies at domain endpoints and outside the domain as well;
for maximal $f$, its right-hand side is empty precisely when $x\notin\dom f$.
\begin{proof}
For fixed $x,u\in\R$, the graph $\gra f\cup\{(x,u)\}$ is monotone
if and only if
\[
 \forall(y,v)\in\gra f:\qquad (x-y)(u-v)\ge0.
\]
For $y<x$ this requires $u\ge v$, for $y>x$ it requires $u\le v$,
and for $y=x$ it imposes no restriction.
Thus the right-hand side of \eqref{eq:maximalvalues} consists exactly of the
values $u$ that can be added at $x$ while preserving monotonicity.
Monotonicity already gives inclusion of $f(x)$ in this set.
If $f$ is maximal, every such value belongs to $f(x)$, proving equality.
Conversely, if equality holds for every $x$, every point of any monotone
extension already belongs to $\gra f$, so no proper monotone extension exists.
\end{proof}

\begin{definition}[Maximal completion and completed restriction]\label{def:completedrestriction}
Let $I=(a,b)\ne\varnothing$ and let $h:I\to\R$ be nondecreasing.
Write $h(x\pm)$ for its one-sided limits, in the extended real line.
Its \emph{maximal completion} $\widehat h$ is
\[
 \widehat h(x)=
 \begin{cases}
  [h(x-),h(x+)],&a<x<b,\\
  (-\infty,h(a+)]\cap\R,&x=a\in\R,\\
  [h(b-),+\infty)\cap\R,&x=b\in\R,\\
  \varnothing,&x\notin\overline I.
 \end{cases}
\]
For nonincreasing $h$, set $\widehat h:=-\widehat{(-h)}$.

For $f\in\Mu\cup\Md$ and a nonempty open interval
$I\subseteq\operatorname{int}\dom f$, the \emph{completed restriction}
$f|_I$ is the unique maximally monotone operator of the same orientation
satisfying
\[
 \forall x\in I:\qquad (f|_I)(x)=f(x),\qquad
 \dom(f|_I)\subseteq\overline I.
\]
\end{definition}
Formula~\eqref{eq:maximalvalues} proves uniqueness of these completions.

\begin{example}[Survival and tail quantile operators]\label{ex:survival}
For a real random variable $X$, the survival function
$t\mapsto\PP(X>t)$ is nonincreasing. Its left and right limits
at $t$ are $\PP(X\geq t)$ and $\PP(X>t)$.
The operator $\SSS[X]$ fills each jump with the interval between
these limits, so it is maximally nonincreasing by
Definition~\ref{def:completedrestriction}.
Its inverse $\TT[X]$ is also maximally nonincreasing, since
inversion preserves maximality and orientation.

For $0<p<1$, $\TT[X](p)$ is a nonempty bounded closed interval.
If $q$ is a measurable selection of $\TT[X]$ and $U$ is uniform
on $(0,1)$, then $q(U)$ has the law of $X$: the set
$\{p\in(0,1):q(p)>t\}$ has length $\PP(X>t)$.
\end{example}

\subsubsection{Sums, products and compositions}
Sums and products of sets are Minkowski operations. For operators, set
\[
 (f+g)(x)=f(x)+g(x),\quad (fg)(x)=f(x)g(x),\quad
 (g\circ f)(x)=\bigcup_{u\in f(x)}g(u).
\]
Note that $(g\circ f)^{-1}=f^{-1}\circ g^{-1}$. 
The domains of the first two are $\dom f\cap\dom g$.

\begin{proposition}[Sums and domain--range characterization]\label{prop:sums}\label{prop:domainrange}
Let $f,g$ be maximal with the same orientation. If $\dom f\cap\dom g\ne\varnothing$, then $f+g$ is maximal with that orientation and
\[
 \ran(f+g)=\ran f+\ran g.
\]
Moreover, a nondecreasing monotone operator $h$ is maximal if and only if
\begin{equation}\label{eq:domainrange}
 \ran(h+\id)\text{ is an interval},\qquad \ran h+\dom h=\R.
\end{equation}
For a nonincreasing operator, replace these conditions by
$\ran(h-\id)$ being an interval and $\ran h-\dom h=\R$.
\end{proposition}
In a general Hilbert space, nonempty domain intersection alone does not ensure maximality of a sum; a standard sufficient condition is $\dom f\cap\operatorname{int}\dom g\ne\varnothing$ \cite{BauschkeCombettes2017}. 

\begin{proof}
First suppose that $f,g$ are nondecreasing, and set
$D=\dom f\cap\dom g$.
If $D$ contains two points, the interval property of the domains gives
\[
 \dom f\cap\operatorname{int}\dom g\ne\varnothing.
\]
Thus $f+g$ is maximal by the sum theorem
\cite{BauschkeCombettes2017}.
If $D=\{c\}$, either one summand has domain $\{c\}$ and takes the
value $\R$ there,
or $c$ is the upper endpoint of one domain and the lower endpoint
of the other. In the latter case, \eqref{eq:maximalvalues} shows that
$f(c)$ and $g(c)$ are oppositely directed rays.
Hence $(f+g)(c)=\R$ in both cases, so $f+g$ is maximal.

To identify its range, fix $w=u+v$ with $u\in f(a)$ and $v\in g(b)$.
Since $f+g$ is maximally monotone, Minty's theorem
(Proposition~\ref{prop:minty}) gives $\ran(\id+f+g)=\R$.
Thus there exists $(x,s)\in\gra(f+g)$ such that
\[
 x+s=\min\{a,b\}+w.
\]
Let us write $s=p+q$ with $(x,p)\in \gra f$ and $(x,q)\in \gra g$. 
Recalling that $(a,u)\in \gra f$ and $(b,v)\in \gra g$, monotonicity of $f$ and $g$ yields
\[
 s>w
 \ \Longrightarrow\ x<\min\{a,b\}
 \ \Longrightarrow\ p\le u \ \text{and} \ q\leq v
 \ \Longrightarrow\ s=p+q\le u+v=w.
\]
Therefore $s\leq w$. 
Using $\max\{a,b\}+w$ instead of $\min\{a,b\}+w$ similarly yields the existence of $s\in \ran(f+g)$ such that $s\geq w$.
Since $\ran(f+g)$ is an interval, $w\in\ran(f+g)$. The reverse inclusion
is immediate, proving $\ran(f+g)=\ran f+\ran g$.
Negating both summands gives the nonincreasing case.

Finally, let $h$ be nondecreasing.
If $h$ is maximal, Minty's theorem gives $\ran(h+\id)=\R$, and
$\ran(h+\id)\subseteq\dom h+\ran h$
proves \eqref{eq:domainrange}.
Conversely, suppose that $\ran(h+\id)$ is an interval and $\dom h+\ran h=\R$.
For $t=x+v\in\dom h+\ran h$, choose $u\in h(x)$ and $y\in h^{-1}(v)$.
If $x<y$, monotonicity gives $u\le v$, hence $x+u\le x+v\le y+v$;
if $x>y$, it gives $u\ge v$, hence $y+v\le x+v\le x+u$;
and if $x=y$, then $t=y+v$ is itself one of the bounding values.
Thus
\[
 \min\{x+u,y+v\}\le t=x+v\le\max\{x+u,y+v\}.
\]
Both bounding values belong to $\ran(h+\id)$, so $t\in\ran(h+\id)$.
Hence $\ran(h+\id)=\dom h+\ran h=\R$,
and Minty's theorem proves maximality.
Applying this characterization to $-h$ gives the nonincreasing case.
\end{proof}
\begin{proposition}[Composition]\label{prop:composition}
Let $f,g$ be maximal with orientations $\sigma,\tau\in\{1,-1\}$, where $1$ denotes nondecreasing, and suppose $g\circ f\ne\varnothing$. The composite is monotone with the expected orientation $\varepsilon=\sigma\tau$ if and only if
\begin{equation}\label{eq:rectangle}
 \forall z\in\R:\qquad
 f^{-1}(z)\text{ or }g(z)\text{ has at most one element}.
\end{equation}
Under this condition,
\begin{equation}\label{eq:compositionmax}
 g\circ f\text{ is maximal in orientation }\varepsilon
 \quad\Longleftrightarrow\quad \dom f+\varepsilon\ran g=\R.
\end{equation}
\end{proposition}
\begin{proof}
Distinct intermediate values order the endpoints of the composite in orientation $\sigma\tau$. Equal intermediate values cause no obstruction precisely when \eqref{eq:rectangle} holds: otherwise $f^{-1}(z)\times g(z)$ contains a nondegenerate rectangle, which is not monotone in either orientation. Reindexing the intermediate variable gives
\[
 \ran\bigl(\id+\varepsilon(g\circ f)\bigr)
 =\ran(f^{-1}+\varepsilon g)
 =\dom f+\varepsilon\ran g.
\]
The two summands have orientation $\sigma$ and nonempty common domain. Apply Proposition~\ref{prop:sums} and Minty's criterion.
\end{proof}
For example, take $f(0)=(-\infty,0]$; for $x>0$, set $f(x)=\{0\}$, and for $x<0$, set $f(x)=\varnothing$. Then $f\in\Mu$ but, for every $x\geq0$, $(f\circ f)(x)=(-\infty,0]$, so the composite is not monotone. The qualification in \eqref{eq:rectangle} cannot be suppressed.

To study nonnegative products, we first record a consequence of the composition criterion.
For a nonnegative operator $k$, write
\[
 (\log k)(x)=\{\log u:u\in k(x),\ u>0\}.
\]
\begin{lemma}[Maximality under logarithms]\label{lem:logmaximal}
Let $f:\R\rightrightarrows\R$ be maximally monotone in either orientation.
If $\ran f\subset[0,\infty)$ and there exists $c>0$ with $c\in\ran f$,
then $\log f$ is maximally monotone with the same orientation.
\end{lemma}
If $\ran f=\{0\}$, then $\gra(\log f)=\varnothing$, so $\log f$ is not maximal.
\begin{proof}
The function $\log:(0,\infty)\to\R$ is a single-valued maximal
nondecreasing operator. The positive-range hypothesis makes
$\log\circ f$ nonempty. The rectangle condition in
Proposition~\ref{prop:composition} is automatic, and
$\dom f+\sigma\ran(\log)=\R$, where $\sigma$ is the
orientation of $f$. That proposition proves maximality.
\end{proof}

\begin{proposition}[Nonnegative products]\label{prop:products}
Let $f,g$ be maximal with the same orientation. If
\[
 \ran f\cup\ran g\subset[0,\infty),
\]
then $fg$ is monotone with that orientation, and
\begin{equation}\label{eq:productcriterion}
 fg\text{ is maximal}
 \quad\Longleftrightarrow\quad
 \dom f\cap\dom g=\R
 \ \text{or}\ \ran(fg)\ne\{0\}.
\end{equation}
When maximal,
\[
 \ran(fg)=\ran f\,\ran g.
\]
\end{proposition}

\begin{proof}[Proof of Proposition~\ref{prop:products}]
Suppose first that $f,g$ are nondecreasing, and put $h=fg$.
Their domains are intervals unbounded below, and $h$ is nondecreasing.

If $\ran h=\{0\}$ and $\dom h\ne\R$, choose $x\notin\dom h$; adding $(x,0)$
preserves monotonicity, so $h$ is not maximal.
Since $\dom h=\dom f\cap\dom g$, this proves the forward implication
in \eqref{eq:productcriterion} by contraposition.

If $\ran h=\{0\}$ and $\dom h=\R$, then $\forall x\in\R$, $h(x) =\{0\}$
which is maximal.
Besides, one factor vanishes identically: otherwise both factors
would be positive at sufficiently large inputs. The range identity
then follows immediately.

We finally prove that $\ran h\ne\{0\}$ implies maximality of $h$.
Since $h$ is nonnegative and has nonempty domain, this condition means
that $h$ has a positive value at some input $c$.
By Lemma~\ref{lem:logmaximal}, $\log f$ and $\log g$ are maximal.

The logarithmic operators have the common input $c$, so
Proposition~\ref{prop:sums} shows that
\[
 \log h=\log f+\log g
\]
is maximal.
Let $(x,w)$ be a point that can be added to $h$ while preserving
monotonicity, that is,
\begin{equation}\label{eq:productcompatibility}
 \forall (y,v)\in\gra h: \qquad (x-y)(w-v)\ge0.
\end{equation}
Since $\dom h$ is unbounded below and $h$ is nonnegative,
\eqref{eq:productcompatibility} with $y<x$ gives $w\ge v\ge0$.
If $w>0$, \eqref{eq:productcompatibility} and strict monotonicity of
the logarithm show that $(x,\log w)$ is compatible with $\log h$,
and its maximality gives $w\in h(x)$.

If $w=0$, \eqref{eq:productcompatibility} applied to the positive value
at $c$ forces $x\le c$; for any $y<x$ and every $v\in h(y)$, it gives $v\le0$,
hence $h(y)=\{0\}$ by nonnegativity.
Moreover, both factor zero sets lie in
$(-\infty,c]\subset\dom h$, hence
\[
 h^{-1}(0)=f^{-1}(0)\cup g^{-1}(0).
\]
This set is closed, since both factor zero sets are closed.
Letting $y\uparrow x$ therefore gives $0\in h(x)$.
Thus $h$ is maximal whenever $\ran h\ne\{0\}$.

Finally, the range identity in Proposition~\ref{prop:sums} gives
\[
 \ran h\cap(0,\infty)
 =\exp\bigl(\ran(\log f)+\ran(\log g)\bigr)
 =(\ran f\,\ran g)\cap(0,\infty).
\]
The identity for the zero sets gives equality also at zero, so
$\ran h=\ran f\,\ran g$.

Reflecting the input proves the nonincreasing case.
\end{proof}
\begin{example}[A zero-degeneracy obstruction]\label{ex:zero}
The increasing maximal operators
\[
 \begin{aligned}
  \forall x\in\R:&\qquad f(x)=\max\{x,0\},\\
  \forall x<0:&\qquad g(x)=-1/x
 \end{aligned}
\]
have nonnegative ranges. Nevertheless, $fg$ is zero on $(-\infty,0)$ and empty elsewhere, so it is not maximal and its range is not $\ran f\,\ran g$. Thus nonnegativity alone does not suffice for product closure.
\end{example}

\subsection{Pointwise resolvent order}\label{sec:order}
To compare the outputs of these operations, we need an order that
remains meaningful when the domains differ. Resolvents put every graph
on the common parameter space $\R$.

\subsubsection{Definition and graph characterization}
We seek an order extending pointwise inequalities of scalar functions, not an order on their slopes. For example, two constants with different heights should be distinguished. This is not the comparison in \cite{BartzBauschkeWang2017}: both signs of a constant difference of resolvents are firmly nonexpansive. The following definition transports ordinary scalar order instead.

\begin{definition}[Pointwise resolvent order]\label{def:order}
For maximal operators of a fixed orientation, define the \emph{pointwise resolvent order} by
\begin{equation}\label{eq:order}
 f\leq g\quad\Longleftrightarrow\quad
 \begin{cases}
 J_f\geq J_g&\text{on }\Mu,\\
 J_f\leq J_g&\text{on }\Md.
 \end{cases}
\end{equation}
All comparisons of functions in this definition are pointwise on $\R$.
\end{definition}
Injectivity of the resolvent representation makes the pointwise resolvent order a partial order on each class. There is no single order on their union with both conventions: for the vertical operators $N_{\{a\}}$ with graphs $\{a\}\times\R$, one has $N_{\{a\}}\leq N_{\{b\}}$ exactly when $a\geq b$ in $\Mu$, and when $a\leq b$ in $\Md$.

\begin{proposition}[Separated-graph characterization]\label{prop:comparison}
For $f,g\in\Mu$,
\begin{equation}\label{eq:separated}
 \begin{aligned}
 f\leq g\quad\Longleftrightarrow\quad&
 \forall(x,u)\in\gra f,\ \forall(y,v)\in\gra g:\ \ x<y\ \Longrightarrow\ u\leq v.
 \end{aligned}
\end{equation}
For $f,g\in\Md$, replace $x<y$ by $x>y$. Consequently, for any $\lambda>0$, \eqref{eq:order} is unchanged when $J_f,J_g$ are replaced by $J_{\lambda f},J_{\lambda g}$.
At every common domain point,
\begin{equation}\label{eq:height}
 f\leq g\quad\Longrightarrow\quad
 \inf f(x)\leq\inf g(x),\qquad \sup f(x)\leq\sup g(x).
\end{equation}
If $f,g$ have the same domain, the endpoint inequalities in \eqref{eq:height} also imply $f\leq g$.
\end{proposition}
This statement relies on the following preliminary result that will be used several times.
For a nonempty interval $I$, denote,
\begin{align}
  I_-=I+(-\infty,0]\qquad \text{and}  \qquad I_+=I+[0,\infty).
  \label{eq:hull_p_m}
\end{align}
\begin{lemma}[Resolvent comparison with an operator value]\label{lem:resolventvalue}
Let $f\in\Mu\cup\Md$, $x\in\dom f$, and $u\in\R$.
If $f$ is nondecreasing ($f\in\Mu$), then
\begin{equation}\label{eq:resolventvalue}
 J_f(x+u)\leq x\Longleftrightarrow u\in f(x)_-,\qquad
 J_f(x+u)\geq x\Longleftrightarrow u\in f(x)_+.
\end{equation}
If $f$ is nonincreasing ($f\in\Md$), then
\begin{equation}\label{eq:resolventvaluedecreasing}
 J_f(x-u)\geq x\Longleftrightarrow u\in f(x)_-,\qquad
 J_f(x-u)\leq x\Longleftrightarrow u\in f(x)_+.
\end{equation}
\end{lemma}
\begin{proof}
First suppose $f\in\Mu$. Put $z=J_f(x+u)$ and $w=x+u-z$.
By Proposition~\ref{prop:minty}, $(z,w)\in\gra f$.
If $z=x$, then $w=u\in f(x)$.
If $z<x$, then $w>u$, and, for every $v\in f(x)$, monotonicity gives $v\geq w>u$.
Thus $z\leq x$ implies $u\in f(x)_-$.
Conversely, $u\in f(x)_-$ means that some $v\in f(x)$ satisfies $u\leq v$; since the resolvent is nondecreasing,
\[
 J_{f}(x+u)\leq J_{f}(x+v)=x.
\]
The second equivalence is proved similarly.
For $f\in\Md$, apply~\eqref{eq:resolventvalue} to $-f\in\Mu$ and the value $-u$.
\end{proof}

\begin{proof}[Proof of Proposition~\ref{prop:comparison}]
Consider the increasing case.
Recalling from~\eqref{eq:minty} that $\gra(f) = \{(J_f(s),s-J_f(s)), s\in \mathbb R\}$ and $\gra(g) = \{(J_g(s),s-J_g(s)), s\in \mathbb R\}$, 
if there exists $s\in\mathbb R$ such that $J_f(s)<J_g(s)$, then $s-J_f(s)>s-J_g(s)$ which would violate \eqref{eq:separated}.

Conversely, assume $J_f\geq J_g$ and consider $(x,u)\in\gra f$, $(y,v)\in\gra g$ satisfying $x<y$.
Set $s=x+u$ and write $(z,w)=(J_g(s),s-J_g(s))\in\gra g$.
Since $J_f(s)=x$, the assumed resolvent inequality gives
\[
 z=J_g(s)\leq J_f(s)=x<y,
 \qquad w=s-z\geq s-x=u.
\]
Monotonicity of $g$, applied to $(z,w)$ and $(y,v)$, yields $w\leq v$.
Hence $u\leq w\leq v$, proving \eqref{eq:separated}.
Reflecting the input gives the decreasing case.
Finally, for $\lambda>0$, the separated-graph condition for $\lambda f,\lambda g$
compares $\lambda u$ and $\lambda v$, and is therefore equivalent to the
condition for $f,g$. This proves scale independence.

For \eqref{eq:height}, fix $x\in\dom f\cap\dom g$ and take any $v\in g(x)$.
Then $J_g(x+v)=x$, so $J_f\geq J_g$ gives $J_f(x+v)\geq x$.
Lemma~\ref{lem:resolventvalue} implies $v\in f(x)_+$, that is, some $u\in f(x)$ satisfies $u\leq v$.
Thus, for every $v\in g(x)$, $\inf f(x)\leq v$, proving $\inf f(x)\leq\inf g(x)$.
For any $u\in f(x)$, similarly,
$J_g(x+u)\leq J_f(x+u)=x$.
Lemma~\ref{lem:resolventvalue} gives $u\in g(x)_-$, so
$u\leq\sup g(x)$. Taking the supremum over $u\in f(x)$ proves
$\sup f(x)\leq\sup g(x)$. Reflection gives the decreasing case. On a common domain, endpoint comparison and the monotonicity of $g$ imply the separated inequalities directly.
\end{proof}
In particular, the pointwise resolvent order is precisely pointwise comparison for continuous monotone scalar functions on $\R$ with the same orientation.

\begin{definition}[Order on intervals]\label{def:intervalorder}
Using the upward and downward hulls defined in~\eqref{eq:hull_p_m}, for nonempty intervals $I,J\subset\R$, define
\begin{equation}\label{eq:intervalorder}
 I\leq J\quad\Longleftrightarrow\quad
 I_+\supseteq J_+\ \text{and}\ I_-\subseteq J_-.
\end{equation}
\end{definition}
On nonempty intervals this is the Egli--Milner order
\cite[Section~3.3]{KeimelPlotkin2017}. It is a partial order, since
$I=I_+\cap I_-$. For closed bounded intervals, $[a,b]\leq[c,d]$ is equivalent to $a\leq c$ and $b\leq d$. The hull formulation also retains endpoint inclusion: for example, $[0,1)\leq(0,1]$, whereas $(0,1)$ and $[0,1]$ are incomparable.

\begin{proposition}[Domains and ranges under comparison]\label{prop:orderdomains}
Let $f,g$ be maximal operators of the same orientation. If $f\leq g$, then, in the interval order of Definition~\ref{def:intervalorder},
\begin{equation}\label{eq:domainrangeorder}
 \ran f\leq\ran g,\qquad
 \begin{cases}
  \dom g\leq\dom f,&f,g\in\Mu,\\
  \dom f\leq\dom g,&f,g\in\Md.
 \end{cases}
\end{equation}
At each $x\in\dom f\cap\dom g$, one also has $f(x)\leq g(x)$ in this interval order.
\end{proposition}
\begin{proof}
If continuous functions $q,r:\R\to\R$ satisfy $q\leq r$, then their ranges are intervals and
\[
 (\ran q)_+\supseteq(\ran r)_+,\qquad
 (\ran q)_-\subseteq(\ran r)_-:
\]
each inclusion follows by comparing $q$ and $r$ at the same argument. Thus $\ran q\leq\ran r$. By Minty's parametrization,
\[
 \dom f=\ran J_f,\qquad
 \ran f=
 \begin{cases}
  \ran(\id-J_f),&f\in\Mu,\\
  \ran(J_f-\id),&f\in\Md.
 \end{cases}
\]
Apply the preceding observation to the resolvents and output coordinates, using Definition~\ref{def:order}. Finally, the nonempty values are closed intervals, so their endpoint inequalities in \eqref{eq:height} give the asserted comparison of operator values.
\end{proof}

\begin{proposition}[Inversion and elementary compatibility]\label{prop:inversion}
For maximal operators,
\begin{equation}\label{eq:inverseorder}
 \begin{array}{ll}
 f,g\in\Mu:&f\leq g\ \Longleftrightarrow\ f^{-1}\geq g^{-1},\\
 f,g\in\Md:&f\leq g\ \Longleftrightarrow\ f^{-1}\leq g^{-1}.
 \end{array}
\end{equation}
If $f\leq g$ and $h$ has their orientation, then, whenever the displayed operations are maximal in that orientation,
\begin{equation}\label{eq:sumorder}
 f+h\leq g+h.
\end{equation}
The analogous inequality holds for products when all ranges are nonnegative.
For $h$ of either orientation, if both composites being compared are maximal in the expected orientation, then
\begin{equation}\label{eq:compositionorder}
 f\circ h\leq g\circ h,\qquad
 \begin{cases}
 h\circ f\leq h\circ g&h\in\Mu,\\
 h\circ f\geq h\circ g&h\in\Md.
 \end{cases}
\end{equation}
\end{proposition}
\begin{proof}
Reflection across the diagonal in \eqref{eq:separated} proves \eqref{eq:inverseorder}. Equivalently,
\[
 \begin{aligned}
  f\in\Mu:&\qquad J_{f^{-1}}=\id-J_f,\\
  f\in\Md:&\qquad J_{f^{-1}}=\id+J_f\circ(-\id).
 \end{aligned}
\]
Add separated inequalities to prove the sum assertion; multiply them for nonnegative products. For right composition, distinct intermediate values permit direct use of \eqref{eq:separated}. If the intermediate values coincide, the monotonicity of both composites and \eqref{eq:rectangle} force both outer operators to be single-valued at that intermediate point; apply \eqref{eq:height}. Inversion gives the left-composition statements. 
\end{proof}

\subsubsection{Survival comparison}
The pointwise resolvent order has a particularly short probabilistic test.
\begin{lemma}[Survival comparison]\label{lem:survivalcomparison}
For a real random variable $X$ and $\alpha\in\Md$,
\begin{equation}\label{eq:survivaltest}
 \SSS[X]\leq\alpha
 \quad\Longleftrightarrow\quad
 \bigl[\forall(t,u)\in\gra\alpha:\quad \PP(X>t)\leq u\bigr].
\end{equation}
Equivalently, $\TT[X]\leq\alpha^{-1}$.
\end{lemma}
\begin{proof}
Necessity follows from graph comparison by letting the input of $\SSS[X]$ decrease to $t$. Conversely, for $x>t$ and $v\in\SSS[X](x)$,
$v\leq\PP(X\geq x)\leq\PP(X>t)\leq u$ and one can 
apply Proposition~\ref{prop:comparison}. The inequality on $\mathcal T[X]$ is a consequence of \eqref{eq:inverseorder}.
\end{proof}
For survival operators this is the usual stochastic comparison, expressed without choosing quantile endpoints; see also \cite[Section~4]{RockafellarRoyset2014}.

\Needspace{18\baselineskip}
\subsection{Infimum and supremum}\label{sec:extrema}

\subsubsection{Arbitrary order-bounded families}
\begin{theorem}[Conditional completeness]\label{thm:lattice}
Let $(f_a)_{a\in A}$ be a nonempty family of maximally monotone operators of a fixed orientation. 
Its infimum exists if and only if the family has a lower bound; its supremum exists if and only if it has an upper bound. 
They are characterized by
\begin{equation}\label{eq:lattice}
\begin{array}{c|cc}
 &J_{\inf_a f_a}&J_{\sup_a f_a}\\ \hline
 \Mu&\sup_aJ_{f_a}&\inf_aJ_{f_a}\\
 \Md&\inf_aJ_{f_a}&\sup_aJ_{f_a}.
\end{array}
\end{equation}
\end{theorem}
\begin{proof}[Proof of Theorem~\ref{thm:lattice}]
A finite-valued pointwise supremum or infimum of nondecreasing $1$-Lipschitz functions is again nondecreasing and $1$-Lipschitz. Finiteness at one point propagates by the common Lipschitz bound.
By Proposition~\ref{prop:minty}, each such envelope is therefore the resolvent of a maximal operator in the specified orientation; the graph formula~\eqref{eq:minty} makes this operator unique.
\end{proof}
Together with \eqref{eq:minty}, these formulas explicitly describe the entire graph of every existing extremum, including its domain endpoints. They require no countability of $A$. The constants $f_n\equiv n$ have no upper bound, so unrestricted completeness would be false.

\begin{lemma}[Inversion and extrema]\label{lem:inverseextrema}
Let $(f_a)_{a\in A}$ be a nonempty family in $\Mu$ or in $\Md$.
If either side of an identity in the table exists, then so does the other, and they are equal.
\[
 \begin{array}{c|cc}
  &(\inf_a f_a)^{-1}&(\sup_a f_a)^{-1}\\ \hline
  \Mu&\sup_a(f_a^{-1})&\inf_a(f_a^{-1})\\
  \Md&\inf_a(f_a^{-1})&\sup_a(f_a^{-1}).
 \end{array}
\]
\end{lemma}

\begin{proof}
By Proposition~\ref{prop:inversion}, inversion is an involution satisfying
\[
 f\le g
 \quad\Longleftrightarrow\quad
 \begin{cases}
  g^{-1}\le f^{-1},&f,g\in\Mu,\\
  f^{-1}\le g^{-1},&f,g\in\Md.
 \end{cases}
\]
It therefore exchanges upper and lower bounds in $\Mu$, and preserves
them in $\Md$. The identities follow from the defining properties
of infimum and supremum.
\end{proof}

Finite infima and suprema satisfy the distributive identities, since
pointwise minimum and maximum of real numbers do. For a finite nonempty
family, every graph point of either extremum belongs to one input graph:
the finite resolvent envelope attains its extremum at every parameter,
and~\eqref{eq:minty} identifies the corresponding graph point.
The universal property of extrema follows directly from pointwise comparison of the resolvents in Definition~\ref{def:order}.
For example, for $f_a,h\in\Mu$, it reads
\begin{align}
  \label{eq:universal_property_sup}
 h=\sup_{a\in A}f_a
 \quad\Longleftrightarrow\quad
 \begin{cases}
  \forall a\in A,\quad h\geq f_a,\\
  \forall g\in\Mu,\quad
  \bigl[\forall a\in A,\ g\geq f_a\bigr]\Longrightarrow g\geq h.
 \end{cases}
\end{align}

The following bounds are further consequences of these universal properties.
\begin{proposition}[Order bounds for extrema]\label{prop:extremaoperations}
Let $A$ be nonempty and $(f_a)_{a\in A}\in\Mu^A$.
Let $\Phi$ assign to each $f\in\Mu$ a set-valued operator $\Phi(f):\R\rightrightarrows\R$.
Assume that:
\[
 \forall f,g\in\Mu\text{ such that }\Phi(f),\Phi(g)\in\Mu,\qquad
 f\leq g\ \Longrightarrow\ \Phi(f)\leq\Phi(g).
\]
Whenever the displayed operators and extrema belong to $\Mu$,
\begin{equation}\label{eq:extremaineq}
 \Phi(\inf_a f_a)\leq\inf_a\Phi(f_a),\qquad
 \sup_a\Phi(f_a)\leq\Phi(\sup_a f_a).
\end{equation}
\end{proposition}
The same conclusions hold with $\Mu$ replaced by $\Md$.
The bounds also hold when $\Phi$ sends one orientation class to the other and preserves the pointwise resolvent order, with extrema taken in the corresponding classes.
For example, $\Phi(f)=f\circ(-\id)$ sends $\Mu$ to $\Md$ and preserves the order and both extrema.
If instead $\Phi$ reverses the pointwise resolvent order, meaning that, whenever the images belong to the target class, $f\leq g$ implies $\Phi(f)\geq\Phi(g)$, the bounds become
\[
 \sup_a\Phi(f_a)\leq\Phi(\inf_a f_a),\qquad
 \Phi(\sup_a f_a)\leq\inf_a\Phi(f_a).
\]
Examples are inversion on $\Mu$, left composition by a decreasing operator whenever the composites are maximal, and $\Phi(f)=-f$.
Inversion on $\Md$ preserves the pointwise resolvent order, as stated in~\eqref{eq:inverseorder}.
\begin{proof}
The inequalities follow from the universal properties of infimum and supremum~\eqref{eq:universal_property_sup} and the assumed preservation or reversal of the pointwise resolvent order.
\end{proof}

\subsubsection{Domains, ranges, and values for finite families}

\begin{definition}[Infimum and supremum of intervals]\label{def:intervalextrema}
Using the notation of Definition~\ref{def:intervalorder}, for a finite nonempty family $(I_a)_{a\in A}$ of nonempty intervals define
\begin{align*}
 \inf_{a\in A}\{I_a\}
   &=\left(\bigcup_{a\in A}(I_a)_+\right)
     \cap\left(\bigcap_{a\in A}(I_a)_-\right),
     \\
 \sup_{a\in A}\{I_a\}&=\left(\bigcup_{a\in A}(I_a)_-\right)
     \cap\left(\bigcap_{a\in A}(I_a)_+\right).
\end{align*}
These intervals are the greatest lower bound and least upper bound, respectively, in the interval order; they need not be members of the original family. Consequently, for nonempty intervals $I,J$,
\begin{equation}\label{eq:intervalorderextrema}
 I\leq J\quad\Longleftrightarrow\quad
 \inf\{I,J\}=I\quad\Longleftrightarrow\quad\sup\{I,J\}=J.
\end{equation}
\end{definition}
For a single interval $I$, $\inf I$ and $\sup I$ denote its scalar endpoints in the extended real line.
The definitions give the hull identities
\begin{align*}
 (\inf_a\{I_a\})_+&=\bigcup_a(I_a)_+,&
 (\inf_a\{I_a\})_-&=\bigcap_a(I_a)_-,\\
 (\sup_a\{I_a\})_+&=\bigcap_a(I_a)_+,&
 (\sup_a\{I_a\})_-&=\bigcup_a(I_a)_-.
\end{align*}
The greatest-lower-bound and least-upper-bound properties now follow directly from \eqref{eq:intervalorder}.

\begin{proposition}[Domains, ranges, and values of finite extrema]\label{prop:finitevalues}\label{prop:extremadomains}
Let $(f_a)_{a\in A}$ be a finite nonempty family in $\Mu$ or in $\Md$. Then
\begin{equation}\label{eq:extremadomains}
\begin{array}{c|cc}
 &\dom(\inf_a f_a)&\dom(\sup_a f_a)\\ \hline
 \Mu&\sup_a\{\dom f_a\}&\inf_a\{\dom f_a\}\\
 \Md&\inf_a\{\dom f_a\}&\sup_a\{\dom f_a\},
\end{array}
\end{equation}
and, in both orientations,
\begin{equation}\label{eq:extremaranges}
 \ran(\inf_a f_a)=\inf_a\{\ran f_a\},\qquad
 \ran(\sup_a f_a)=\sup_a\{\ran f_a\}.
\end{equation}
At every $x\in D:=\bigcap_a\dom f_a$,
\begin{align}\label{eq:finitevalues}
 (\inf_a f_a)(x)&=\inf_a\{f_a(x)\}
   =[\inf_a\inf f_a(x),\ \inf_a\sup f_a(x)]\cap\R,\\
 (\sup_a f_a)(x)&=\sup_a\{f_a(x)\}
   =[\sup_a\inf f_a(x),\ \sup_a\sup f_a(x)]\cap\R.\notag
\end{align}
\end{proposition}
\begin{proof}
For a finite nonempty family of continuous functions $q_a:\R\to\R$ monotone in the same direction, one can show that
\begin{equation}\label{eq:rangeenvelopes}
 \ran(\inf_a q_a)=\inf_a\{\ran q_a\},\qquad
 \ran(\sup_a q_a)=\sup_a\{\ran q_a\}.
\end{equation}
Indeed, put $m(s)=\inf_a q_a(s)$. Its range has hulls
\[
 (\ran m)_+=\bigcup_a(\ran q_a)_+,
 \qquad
 (\ran m)_-=\bigcap_a(\ran q_a)_-.
\]
For the second equality, if $t\in\bigcap_a(\ran q_a)_-$, for each
$a$ choose $s_a$ with $t\leq q_a(s_a)$. Finiteness and the common
direction give a common witness $s=\max_a s_a$ in the nondecreasing
case, or $s=\min_a s_a$ in the nonincreasing case; then $t\leq m(s)$.
The reverse inclusion and the first equality are immediate.
Since $m$ is continuous, $\ran m$ is an interval, so
\[
 \ran m=(\ran m)_+\cap(\ran m)_-
 =\left(\bigcup_a(\ran q_a)_+\right)
  \cap\left(\bigcap_a(\ran q_a)_-\right)
 =\inf_a\{\ran q_a\}.
\]
The supremum identity follows analogously, with the inequalities reversed.

Apply \eqref{eq:rangeenvelopes} to the Minty coordinates used in Proposition~\ref{prop:orderdomains}, and use \eqref{eq:lattice} to obtain \eqref{eq:extremadomains} and \eqref{eq:extremaranges}.

To compute the values, consider first the increasing case and fix $x\in D$ and $u\in\R$.
By Lemma~\ref{lem:resolventvalue}, 
\begin{align*}
 u\in(\inf_a f_a)(x)
 &\Longleftrightarrow \sup_a J_{f_a}(x+u)=x\\
 &\Longleftrightarrow
 \bigl[\forall a,\ u\in f_a(x)_-\bigr]
 \ \text{and}\ \bigl[\exists a,\ u\in f_a(x)_+\bigr]\\
 &\Longleftrightarrow
 u\in\left(\bigcap_a f_a(x)_-\right)
       \cap\left(\bigcup_a f_a(x)_+\right)
 =\inf_a\{f_a(x)\}.
\end{align*}
Here the first equivalence uses~\eqref{eq:lattice}, and the last equality is Definition~\ref{def:intervalextrema}.
The infimum resolvent envelope similarly gives the supremum of the values.
For decreasing operators, the same argument uses~\eqref{eq:resolventvaluedecreasing}.
Since the nonempty values of maximal operators are closed intervals, their interval extrema have the endpoint expressions in~\eqref{eq:finitevalues}.
\end{proof}
Figure~\ref{fig:decreasingextrema} shows the supremum and infimum of three maximal decreasing operators.

\begin{figure}[htbp]
 \centering
 \includegraphics[width=\textwidth]{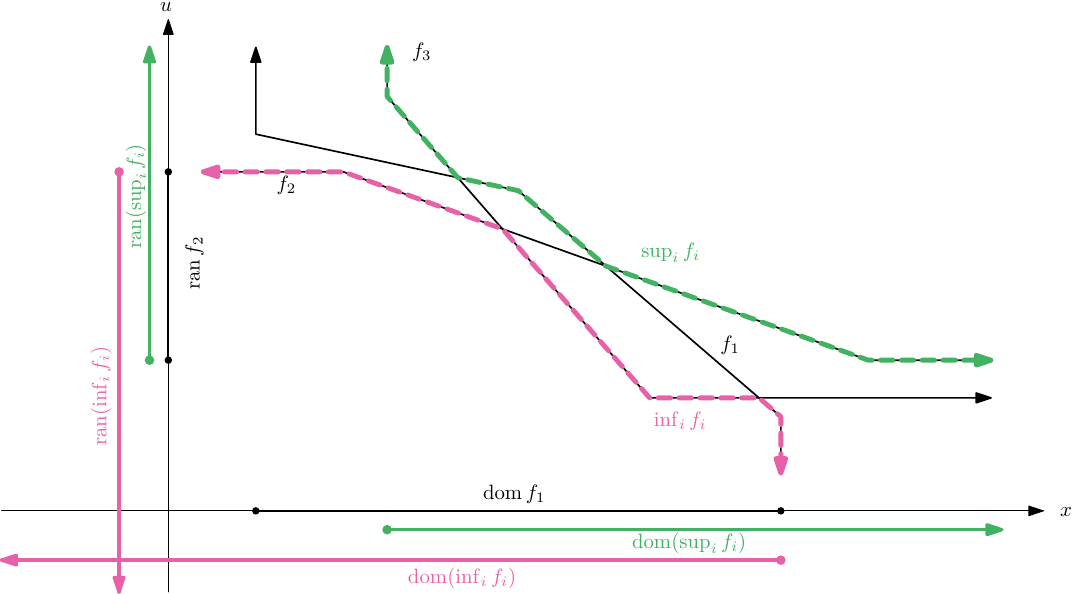}
 \caption{Three maximal decreasing operators $f_1,f_2,f_3$ (black), their supremum (green, dashed) and infimum (pink, dashed).}
 \label{fig:decreasingextrema}
\end{figure}

When domains differ, the pointwise formulas apply on their intersection. For example, in $\Mu$, $\inf(N_{\{0\}},N_{\{1\}})=N_{\{1\}}$, although $N_{\{0\}}(1)=\varnothing$. Nor do these pointwise formulas extend unchanged to infinite families. For example, let $f_n$ be the increasing step from $0$ to $1$ at $1/n$, completed at its jump. Then $(\sup_n f_n)(0)=[0,1]$, although every $f_n(0)=\{0\}$. Formula~\eqref{eq:lattice} automatically performs the necessary completion.

\begin{proposition}[Lattice median]\label{prop:latticemedian}
Let $r\ge0$ be an integer, set $n=2r+1$ and
$[n]:=\{1,\ldots,n\}$, and let $f_1,\ldots,f_n$ be
maximal operators of the same orientation. The two finite lattice
expressions
\begin{equation}\label{eq:median}
 m:=\operatorname{med}(f_1,\ldots,f_n)
  =\inf_{\substack{J\subseteq[n]\\|J|=r+1}}\sup_{i\in J}f_i
  =\sup_{\substack{J\subseteq[n]\\|J|=r+1}}\inf_{i\in J}f_i
\end{equation}
coincide and define the \emph{operator median}, a maximal operator
satisfying, simultaneously for every $z\in\R$ and $\lambda>0$,
\[
 J_{\lambda m}(z)=\operatorname{med}_{1\le i\le n}J_{\lambda f_i}(z).
\]
If at least $r+1$ of the $f_i$ lie between maximal operators $\ell$ and $h$ of that orientation, then $\ell\le m\le h$.
\end{proposition}
\begin{proof}
The scalar odd median is both $\min_{|I|=r+1}\max_{i\in I}$ and $\max_{|I|=r+1}\min_{i\in I}$. The finite resolvent formulas in Theorem~\ref{thm:lattice} therefore identify both expressions in \eqref{eq:median}; reversing scalar order leaves the odd median unchanged. Positive output scaling is an isomorphism for the pointwise resolvent order, so the identity holds at every scale for the same $m$. The scalar majority property applied to the resolvents gives the final assertion.
\end{proof}

\subsubsection{Extrema and operations}
We now combine finite extrema with sums and products.
\begin{proposition}[Finite algebraic identities]\label{prop:finiteextremaoperations}
Let $A$ be finite and nonempty, $(f_a)_{a\in A}\in\Mu^A$, and $h\in\Mu$.
If, for every $a\in A$, $h+f_a\in\Mu$, then
\begin{equation}\label{eq:finitealgebraic}
 h+(\inf_a f_a)=\inf_a(h+f_a),\qquad
 h+(\sup_a f_a)=\sup_a(h+f_a).
\end{equation}
For products, they hold if $\ran h\subset[0,\infty)$ and, for every $a\in A$, $hf_a\in\Mu$.
All these statements also hold with $\Mu$ replaced by $\Md$.
\end{proposition}
\begin{proof}
Put $D=\dom h$. For addition, maximality of each $h+f_a$ gives
$D\cap\dom f_a\ne\varnothing$.
Intersection with $D$ therefore commutes with the finite interval extrema of these domains.
Using~\eqref{eq:extremadomains}, we obtain
\[
 \begin{aligned}
 \dom\bigl(\inf_a(h+f_a)\bigr)
 &=\sup_a\{D\cap\dom f_a\}
 =D\cap\dom(\inf_a f_a),\\
 \dom\bigl(\sup_a(h+f_a)\bigr)
 &=\inf_a\{D\cap\dom f_a\}
 =D\cap\dom(\sup_a f_a).
 \end{aligned}
\]
For each identity, let $E$ be this common domain.
For $x\in E$, put $A_x=\{a\in A:x\in\dom f_a\}$.
To justify omitting indices with $f_a(x)=\varnothing$, let $k=\inf_a f_a$, so
$J_k=\max_a J_{f_a}$. At a graph point with abscissa $x$, this finite
maximum is attained. If $x\notin\dom f_a=\ran J_{f_a}$, continuity
of $J_{f_a}$ and its value below $x$ at that parameter imply that,
for every $s\in\R$, $J_{f_a}(s)<x$. Removing this member therefore
leaves the equation $\max_a J_{f_a}(s)=x$ unchanged. For
$k=\sup_a f_a$, use the finite minimum and $J_{f_a}(s)>x$ instead.
Thus~\eqref{eq:finitevalues} applies to the nonempty values indexed by
$A_x$, even when the domains differ.
Addition by the fixed interval $h(x)$ commutes with their infimum and supremum,
proving both identities at every point of their domains.

For products, the same argument comparing domains and values applies, since $\dom(hf_a)=D\cap\dom f_a$.
For nonempty closed intervals $H\subset[0,\infty)$ and $I_a$, interval-endpoint arithmetic gives
\[
 H\bigl(\inf_a\{I_a\}\bigr)=\inf_a\{HI_a\},\qquad
 H\bigl(\sup_a\{I_a\}\bigr)=\sup_a\{HI_a\}.
\]
For $x\in E$, take $H=h(x)$ and, for $a\in A_x$, $I_a=f_a(x)$.
This proves the product identities, including at domain endpoints.
Input reflection gives the identities in $\Md$.
\end{proof}

\begin{proposition}[Composition and extrema]\label{prop:extremacomposition}
Let $A$ be finite and nonempty, let $(f_a)_{a\in A}$ belong to one
maximal orientation class, and let $h$ be maximal with specified
orientation.
If every $f_a\circ h$ is maximal in the expected orientation, then
\begin{equation}\label{eq:finitecomposition}
 (\inf_a f_a)\circ h=\inf_a(f_a\circ h),\qquad
 (\sup_a f_a)\circ h=\sup_a(f_a\circ h).
\end{equation}
If instead every $h\circ f_a$ is maximal in the expected orientation,
then
\[
 \begin{array}{c|cc}
  &h\circ(\inf_a f_a)&h\circ(\sup_a f_a)\\ \hline
  h\in\Mu&\inf_a(h\circ f_a)&\sup_a(h\circ f_a)\\
  h\in\Md&\sup_a(h\circ f_a)&\inf_a(h\circ f_a).
 \end{array}
\]
\end{proposition}

\begin{proof}
For right composition, input and output reflection reduce the proof
to $f_a,h\in\Mu$. Put
\[
 g_a=f_a\circ h,\qquad f=\inf_a f_a,\qquad g=\inf_a g_a.
\]
The finite resolvent formula \eqref{eq:lattice} gives
\[
 \gra f\subseteq\bigcup_a\gra f_a.
\]
We first show $g\subseteq f\circ h$.
Fix $s\in\R$ and choose, for each $a\in A$,
\[
 x_a=J_{g_a}(s),\qquad u_a=s-x_a,\qquad
 z_a\in h(x_a)\cap f_a^{-1}(u_a).
\]
Choose an index $b$ with $x_b$ largest, and $z_b$ largest among
the tied indices. Monotonicity of $h$ then gives
\[
 x_b=\max_a x_a=J_g(s),\qquad
 z_b=\max_a z_a,\qquad
 u_b=\min_a u_a.
\]
Since each resolvent is nondecreasing and $1$-Lipschitz,
\[
 \begin{aligned}
 J_{f_a}(z_b+u_b)
 &\le J_{f_a}(z_b+u_a)
 \le J_{f_a}(z_a+u_a)+(z_b-z_a)
 =z_b.
 \end{aligned}
\]
Equality holds for $a=b$, so
\[
 J_f(z_b+u_b)=\max_aJ_{f_a}(z_b+u_b)=z_b.
\]
Hence $u_b\in f(z_b)$ and, since $z_b\in h(x_b)$,
\[
 (x_b,u_b)\in\gra(f\circ h).
\]
As $s$ varies, these points parametrize $\gra g$.
Thus $g\subseteq f\circ h$, and in particular $f\circ h$ is nonempty.

We apply Proposition~\ref{prop:composition} to $h$ and $f$.
For any $z\in\R$ such that $h^{-1}(z)$ contains more than one point,
\eqref{eq:rectangle} for each $g_a$ shows that every $f_a(z)$ contains at most one point.
Since $f(z)$ is an interval contained in their finite union, it too contains at most one point.
Thus $h$ and $f$ satisfy~\eqref{eq:rectangle}.
Proposition~\ref{prop:composition} therefore shows that $f\circ h$ is nondecreasing.
Since $g$ is maximal and $g\subseteq f\circ h$, we conclude that $g=f\circ h$.
Choosing the smallest $x_b$, then the smallest $z_b$ among ties,
and reversing the resolvent inequalities proves the supremum identity.

For left composition, first suppose that $f_a\in\Mu$.
Lemma~\ref{lem:inverseextrema}, the right-composition identities,
and inversion of compositions give
\[
 \begin{aligned}
 \bigl(h\circ(\inf_a f_a)\bigr)^{-1}
 &=\bigl(\sup_a f_a^{-1}\bigr)\circ h^{-1}\\
 &=\sup_a\bigl(f_a^{-1}\circ h^{-1}\bigr)\\
 &=\sup_a(h\circ f_a)^{-1}\\
 &=
 \begin{cases}
  \bigl(\inf_a(h\circ f_a)\bigr)^{-1},&h\in\Mu,\\
  \bigl(\sup_a(h\circ f_a)\bigr)^{-1},&h\in\Md.
 \end{cases}
 \end{aligned}
\]
Inverting proves the asserted formula.
For $f_a\in\Md$, the same calculation uses infima instead of
suprema in the middle lines and gives the same conclusion.
Exchanging infima and suprema proves the remaining identities.
\end{proof}

\subsubsection{Parallel operations for survival bounds and multilevel profiles}
\label{subsec:multilevel}

Survival bounds combine thresholds at a common tail level. Inversion
expresses this calculation using ordinary sums and products of the values
of inverse operators. We introduce parallel notation here for this purpose; its
algebraic properties follow from the ordinary calculus already proved.
For operators $f,g$, define
\begin{equation}\label{eq:parallel}
 f\boxplus g=(f^{-1}+g^{-1})^{-1},\qquad
 f\boxtimes g=(f^{-1}g^{-1})^{-1}.
\end{equation}
Both operations are associative and commutative as operations on graphs,
and their range is $\ran f\cap\ran g$. If $f,g$ are maximal with the
same orientation and this intersection is nonempty, Proposition~\ref{prop:sums}
applied to their inverses gives
\[
 f\boxplus g\text{ maximal},\qquad
 \dom(f\boxplus g)=\dom f+\dom g.
\]
This is the scalar parallel sum; see also \cite{Passty1986}.
Figure~\ref{fig:parallelsum} compares ordinary addition of outputs at a
common input with parallel addition of inputs at a common output.
For nonnegative domains, maximality of $f\boxtimes g$ is checked by
applying Proposition~\ref{prop:products} to $f^{-1},g^{-1}$.
Nonnegativity alone does not guarantee it.

\begin{figure}[H]
 \centering
 \includegraphics[width=0.485\textwidth]{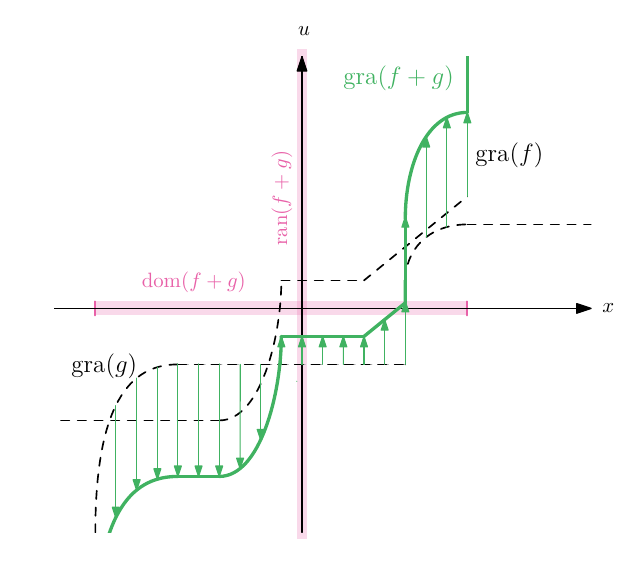}\hfill%
 \includegraphics[width=0.485\textwidth]{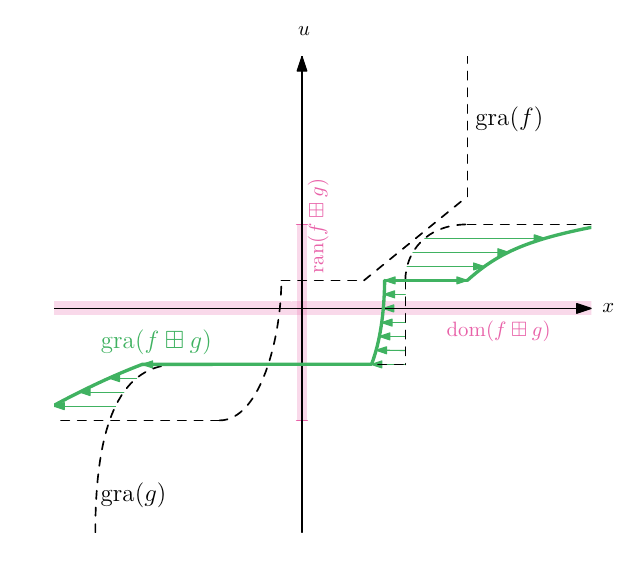}
 \caption{Ordinary sum $f+g$ (left) and parallel sum $f\boxplus g$
 (right) of two maximal increasing operators $f,g$ (black, dashed).
 Green arrows illustrate addition of outputs at a common input for the
 ordinary sum and addition of inputs at a common output for the parallel
 sum. Pink axis intervals indicate the resulting domains and ranges.}
 \label{fig:parallelsum}
\end{figure}

\begin{proposition}[Sums and nonnegative products]\label{prop:probability}
Let $X_1,\ldots,X_n$ be real random variables on one probability space,
with $\SSS[X_i]\leq\alpha_i\in\Md$ and
$\bigcap_i\ran\alpha_i\ne\varnothing$. Then
\begin{equation}\label{eq:probsum}
 \SSS\!\left[\sum_{i=1}^nX_i\right]
 \leq n(\alpha_1\boxplus\cdots\boxplus\alpha_n).
\end{equation}
If additionally $X_i\geq0$ almost surely,
$\dom\alpha_i\subset[0,\infty)$, and
$\alpha_1^{-1}\cdots\alpha_n^{-1}$ is maximal, then
\begin{equation}\label{eq:probproduct}
 \SSS\!\left[\prod_{i=1}^nX_i\right]
 \leq n(\alpha_1\boxtimes\cdots\boxtimes\alpha_n).
\end{equation}
\end{proposition}
\begin{proof}
For $(t,u)\in\gra(\alpha_1\boxplus\cdots\boxplus\alpha_n)$,
choose $t_i\in\alpha_i^{-1}(u)$ with $t=\sum_i t_i$.
Lemma~\ref{lem:survivalcomparison} and the union bound give
\[
 \PP\!\left(\sum_iX_i>t\right)
 \leq\sum_i\PP(X_i>t_i)\leq nu.
\]
Applying the same lemma to $n(\alpha_1\boxplus\cdots\boxplus\alpha_n)$
proves \eqref{eq:probsum}. For products, choose nonnegative
$t_i\in\alpha_i^{-1}(u)$ with $t=\prod_i t_i$ and use
\[
 \left\{\prod_iX_i>\prod_i t_i\right\}
 \subseteq\bigcup_i\{X_i>t_i\}.
\]
\end{proof}
For integrable sums, Hardy averaging gives the complementary estimate
\eqref{eq:hardytailsum}, without this union-bound factor.

\begin{corollary}[Centered sums and products]\label{cor:centered}
Suppose $\SSS[|X-c|]\leq\alpha$ and $\SSS[|Y-d|]\leq\beta$, where
$\alpha,\beta\in\Md$ have nonnegative domains and
$\ran\alpha\cap\ran\beta\ne\varnothing$.
Whenever all the displayed parallel expressions are maximal,
\[
 \SSS[|X+Y-c-d|]\leq2(\alpha\boxplus\beta).
\]
For $c,d\ne0$, one also has
\begin{equation}\label{eq:centeredproduct}
 \SSS[|XY-cd|]\leq
 2\left((\alpha\boxtimes\beta)
 \boxplus\alpha\circ\frac{\id}{|d|}
 \boxplus\beta\circ\frac{\id}{|c|}\right).
\end{equation}
If $d=0$, omit the term containing $\alpha\circ(\id/|d|)$;
if $c=0$, omit the term containing $\beta\circ(\id/|c|)$.
\end{corollary}
\begin{proof}
On the event $|X-c|\leq r$, $|Y-d|\leq s$, one has
\[
 |X+Y-c-d|\leq r+s,\qquad
 |XY-cd|\leq rs+|d|r+|c|s.
\]
For $r\in\alpha^{-1}(u)$ and $s\in\beta^{-1}(u)$, the exceptional
probability is at most $2u$. If $I,J$ are closed intervals of nonnegative
numbers, the image of $I\times J$ under
$(r,s)\mapsto rs+|d|r+|c|s$ is the interval
$IJ+|d|I+|c|J$: the lower and upper endpoints, including their
attainment or divergence, agree. Inverting this identity and applying
Lemma~\ref{lem:survivalcomparison} proves \eqref{eq:centeredproduct}.
Only the two original exceptional events are used, so the coefficient
is $2$ even though the threshold contains three terms.
\end{proof}

The next estimate converts a parallel sum into explicit scalar envelopes.
\begin{proposition}[Weighted parallel-sum envelopes]\label{prop:envelopes}
Let $\alpha_1,\ldots,\alpha_n\in\Md$ have common range $D$.
For $\theta_i>0$ with $\sum_i\theta_i=1$,
\begin{equation}\label{eq:envelopes}
 \inf_i\bigl(\alpha_i\circ(\theta_i\id)\bigr)
 \leq\alpha_1\boxplus\cdots\boxplus\alpha_n
 \leq\sup_i\bigl(\alpha_i\circ(\theta_i\id)\bigr).
\end{equation}
\end{proposition}
\begin{proof}
Put $q_i=\alpha_i^{-1}$ and $q=\sum_i q_i$.
All these inverse operators have domain $D$. The endpoints of
$q(u)=\sum_i\theta_i(q_i(u)/\theta_i)$ lie between the corresponding
endpoints of the finite infimum and supremum of $q_i(u)/\theta_i$.
The formulas for the domains and values in \eqref{eq:extremadomains} and
\eqref{eq:finitevalues} therefore give
\[
 \inf_i(q_i/\theta_i)\leq q\leq\sup_i(q_i/\theta_i).
\]
Inversion preserves order in $\Md$ and commutes with these extrema by
Lemma~\ref{lem:inverseextrema}. Since
$(q_i/\theta_i)^{-1}=\alpha_i\circ(\theta_i\id)$, this proves
\eqref{eq:envelopes}, including the values at domain boundaries.
\end{proof}

\begin{example}[Multilevel concentration for a product]
\label{ex:productmultilevel}
Let $n\geq1$ and let $X_1,\ldots,X_n$ be real random variables.
Let $a_i\in\R\setminus \{0\}$, $\sigma_i>0$,
and $\alpha\in\Md$ satisfy
\begin{equation}\label{eq:positiveprofile}
 \begin{aligned}
 &\dom\alpha=[0,\infty),\\
 &\forall t>0:\qquad \alpha(t)\subset[0,1],\\
 &\forall i\in[n]:\qquad
   \SSS[|X_i-a_i|]\leq\alpha\circ\frac{\id}{\sigma_i}.
 \end{aligned}
\end{equation}
We bound the deviation $Y:=|\prod_{i=1}^nX_i-\prod_{i=1}^na_i|$.
For every $r\geq0$, expanding the product shows that $Y\leq P(r)$
whenever, for every $i\in[n]$, $|X_i-a_i|\leq\sigma_i r$, where
\[
 \begin{aligned}
 P(r)&:=\prod_{i=1}^n(|a_i|+\sigma_i r)-\prod_{i=1}^n|a_i|
       =\sum_{k=1}^n c_k r^k,\\
 c_k&:=\sum_{\substack{A\subseteq\{1,\ldots,n\}\\|A|=k}}
            \prod_{i\in A}\sigma_i\prod_{i\notin A}|a_i|.
 \end{aligned}
\]
Since $c_n=\prod_i\sigma_i>0$, $P$ is an increasing bijection of
$[0,\infty)$. The union bound therefore gives
$\SSS[Y]\leq n(\alpha\circ P^{-1})$.

We now rewrite this bound using the parallel-sum calculus.
All inverse powers below are ordinary maps on $[0,\infty)$;
their composites with $\alpha$ include the completed value at zero.
Since $P$ is the sum of its monomials,
\[
 P^{-1}=\mathop{\boxplus}_{k\in [n]}
                 \left(\frac{\id}{c_k}\right)^{1/k}.
\]
For these maps, left composition by $\alpha$ distributes over the
parallel sum. Proposition~\ref{prop:envelopes} then gives, for
$\theta_k>0$ with $\sum_{k\in [n]}\theta_k=1$,
\begin{equation}\label{eq:productmultilevelprofile}
 \begin{aligned}
 \SSS[Y]
 &\leq n(\alpha\circ P^{-1})
   =n\left(\mathop{\boxplus}_{k\in [n]}
       \left[\alpha\circ
         \left(\frac{\id}{c_k}\right)^{1/k}\right]\right)\\
 &\leq n\sup_{k\in [n]}
       \left(\alpha\circ
           \left(\frac{\theta_k\id}{c_k}\right)^{1/k}\right).
 \end{aligned}
\end{equation}
The distributive identity also holds at jumps: if
$I=\alpha^{-1}(u)$ is a nonempty interval, the inverses of both
sides take the same value $P(I)=\sum_{k\in [n]}c_k I^k$ at $u$.

For example, for $t>0$ set $\alpha(t)=e^{-t^2}$ and take its maximal
completion. If $d$ is the number of positive coefficients $c_k$, choosing
$\theta_k=1/d$ yields
\begin{equation}\label{eq:multilevel}
 \forall t>0:\qquad
 \PP\!\left(\left|\prod_{i=1}^nX_i-\prod_{i=1}^na_i\right|>t\right)
 \leq\min\left\{1,\,
 n\exp\!\left[-\min_{k\in [n]}
                   \left(\frac{t}{d c_k}\right)^{2/k}\right]\right\}.
\end{equation}
Each degree $k$ in the product expansion gives a power $2/k$.
The factor $n$ comes from the original marginal bounds; replacing
the parallel sum by a supremum adds no further probability factor.

We recover with this elementary example the classical multilevel profile
\cite{BoucheronBousquetLugosiMassart2005,Latala2006,
AdamczakWolff2015,BobkovGotzeSambale2019,
gotze2021concentration_poly,gotze2021concentration}.
Here it concerns mere product of random variables and follows directly from the given marginal tail bounds by
distributivity of composition over the parallel sum, followed by
a supremum bound for that sum.
The cited results concern more general non-Lipschitz functions of random vectors and rely
on concentration of measure assumptions.
\end{example}

\subsection{Integration of maximally monotone operators}
\label{sec:integration}

We use the Aumann integral to integrate both a single operator over
an interval and a family of operators.

\begin{definition}[Aumann integral]\label{def:aumannintegral}
Let $(\Theta,\mathcal A,\mu)$ be a measure space and let
$h:\Theta\rightrightarrows\R$ be measurable with nonempty closed
values. Its \emph{Aumann integral} \cite{Aumann1965} is
\begin{equation}\label{eq:aumann}
 \int_\Theta^{\mathrm A}h\,d\mu
 :=\left\{\int_\Theta s\,d\mu:
       s\in L^1(\mu),\quad \text{for almost every }\theta:
       \ s(\theta)\in h(\theta)\right\}.
\end{equation}
Here measurability means that, for every open set $O\subseteq\R$,
$\{\theta:h(\theta)\cap O\ne\varnothing\}\in\mathcal A$.
For interval-valued maps, this is equivalent to measurability of the
two extended-real endpoints. If $h=[\ell,u]$ with
$\ell,u\in L^1(\mu)$, then
\begin{equation}\label{eq:aumanninterval}
 \int_\Theta^{\mathrm A}h\,d\mu
 =\left[\int_\Theta\ell\,d\mu,\int_\Theta u\,d\mu\right].
\end{equation}
\end{definition}
The last identity follows by integrating the endpoint bounds;
constant convex combinations of $\ell$ and $u$ attain every value
between the two integrals.

\begin{definition}[Scalar integration]\label{def:area}
Let $f\in\Mu\cup\Md$ have nonempty domain interior $I$.
For $a<b$ with $(a,b)\subseteq I$, choose a measurable selection
$s$ of $f$ on $(a,b)$ and write
\[
 \int_a^b f(t)\,dt:=\int_a^b s(t)\,dt,
 \qquad
 \int_a^b |f(t)|\,dt:=\int_a^b |s(t)|\,dt.
\]
We allow infinite endpoints and extended scalar integrals whenever
the positive and negative parts are not both infinite.
We call $f$ integrable on $(a,b)$ if $\int_a^b|f(t)|\,dt<\infty$.
Adding or removing endpoints does not change these integrals.
For reversed limits, set $\int_b^a f=-\int_a^b f$, and set
$\int_a^a f=0$. Other Borel scalar expressions in $f$ are integrated
through the same selection convention whenever their integrals exist.

We call $f$ \emph{integrable} if it is integrable on its domain
interior $I$, that is, if $\int_I|f(t)|\,dt<\infty$.
\end{definition}
Indeed, on $I$ the values of $f$ are bounded intervals with monotone
endpoints. The interiors of its nonsingleton values at distinct
arguments are disjoint, so $f$ is multivalued at at most countably
many points. All selections therefore agree almost everywhere.
In the integrable case,
\[
 \int_{(a,b)}^{\mathrm A}f(t)\,dt
 =\left\{\int_a^b f(t)\,dt\right\}.
\]
Thus scalar integration uses the value of a singleton Aumann integral;
extended scalar integrals use the ordinary Lebesgue convention.

\begin{proposition}[Inverse-area identity]\label{prop:area}
Let $\alpha\in\Md$ have graph contained in $[0,\infty)^2$, with
$(0,\infty)\subseteq\dom\alpha$ and
$\lim_{t\to\infty}\sup\alpha(t)=0$.
Then $\alpha^{-1}$ is defined on $(0,\infty)$ and
\begin{equation}\label{eq:inversearea}
 \int_0^\infty\alpha(t)\,dt
 =\int_0^\infty\alpha^{-1}(p)\,dp\in[0,\infty].
\end{equation}
\end{proposition}
\begin{proof}
Maximality makes $\ran\alpha$ an interval. If it were bounded above,
a point with negative abscissa and sufficiently large ordinate could
be added to the graph, contradicting maximality. Decay to zero gives
$\inf\ran\alpha=0$, so every positive inverse level is attained.
In the region below the graph in the positive quadrant,
the vertical section at $t>0$ has length $\sup\alpha(t)$.
By monotonicity, the horizontal section at $p>0$ has length
$\sup\alpha^{-1}(p)$. Tonelli's theorem gives the identity.
\end{proof}

For a nonnegative random variable $Y$, complete $\SSS[Y]$ on the
positive half-line. Its inverse agrees with $\TT[Y]$ on $(0,1)$
and is zero above $1$. The inverse-area identity and the layer-cake
formula give
\begin{equation}\label{eq:expectationarea}
 \E Y=\int_0^\infty\SSS[Y](t)\,dt
     =\int_0^1\TT[Y](p)\,dp.
\end{equation}
For signed $X\in L^1$, the quantile representation in
Example~\ref{ex:survival} also gives
$\E X=\int_0^1\TT[X](p)\,dp$.

\begin{corollary}[Moments]\label{cor:moments}
For $Y\geq0$ and $r>0$,
\begin{equation}\label{eq:momentidentities}
 \E [Y^r]
 =r\int_0^\infty t^{r-1}\SSS[Y](t)\,dt
 =\int_0^1\bigl(\TT[Y](p)\bigr)^r\,dp.
\end{equation}
\end{corollary}
If $X$ is real-valued, $c\in\R$, and
$\SSS[|X-c|]\leq\alpha\in\Md$, where
$(0,\infty)\subseteq\dom\alpha$ and the values of $\alpha$ are
nonnegative there, set $\beta=\inf\{\alpha,1\}$ in the operator order.
Then
\begin{equation}\label{eq:moment}
 \E|X-c|^r
 \leq\int_0^\infty\beta(t^{1/r})\,dt
 =r\int_0^\infty t^{r-1}\beta(t)\,dt.
\end{equation}
All these integrals may take the value $+\infty$.

\begin{proof}
Apply the layer-cake formula to $Y^r$ and substitute $t=u^r$.
The quantile representation gives the last integral in
\eqref{eq:momentidentities}. For the bound, use
$\SSS[|X-c|]\leq\beta$ and integrate.
\end{proof}

\begin{corollary}[Change of pivot to the mean]\label{prop:pivots}
Let $X$ be real-valued and $c\in\R$.
If $\SSS[|X-c|]\leq\beta\in\Md$, with nonnegative values on
$(0,\infty)\subseteq\dom\beta$, and
$M:=\int_0^\infty\beta(t)\,dt<\infty$, then $X\in L^1$ and
\begin{equation}\label{eq:meanpivot}
 |\E X-c|\leq M,\qquad
 \SSS[|X-\E X|]\leq\beta\circ(\id-M).
\end{equation}
\end{corollary}
\begin{proof}
Equation~\eqref{eq:expectationarea} gives $\E|X-c|\leq M$.
Now use $|X-\E X|\leq|X-c|+M$ and translate the survival bound.
\end{proof}
For an independent copy $X'$, a median $m$ of $X$, and any $c\in\R$,
we also have, for $t>0$,
\begin{equation}\label{eq:independentcopy}
 \PP(|X-m|>t)\leq2\PP(|X-X'|>t),\qquad
 \PP(|X-X'|>t)\leq2\PP(|X-c|>t/2).
\end{equation}
For the first bound, use the disjoint events
$\{X>m+t,X'\leq m\}$ and $\{X<m-t,X'\geq m\}$.
The second follows from the triangle inequality and a union bound.

\begin{definition}[Integration of a family]\label{def:familyintegral}
\label{def:measurablefamilyintegral}
Let $(\Theta,\mathcal A,\mu)$ be a finite measure space with
$\mu(\Theta)>0$. Let $(f_\theta)_{\theta\in\Theta}$ be maximal
operators of one orientation with common nonempty domain interior $I$.
For each $x\in I$, suppose that the endpoint maps
$\theta\mapsto\inf f_\theta(x)$ and
$\theta\mapsto\sup f_\theta(x)$ are measurable and integrable.
Define $\int_\Theta f_\theta\,d\mu(\theta)$ by taking the pointwise
Aumann integral on $I$ and then its maximal completion on
$\overline I$.
\end{definition}

\begin{proposition}[Aumann integration of a monotone family]
\label{thm:familyintegral}
The integral in Definition~\ref{def:familyintegral} is well defined,
with values on $I$ given by
\begin{equation}\label{eq:integralvalues}
 \left(\int_\Theta f_\theta\,d\mu\right)(x)
 =\int_\Theta^{\mathrm A}f_\theta(x)\,d\mu
 =\left[\int_\Theta\inf f_\theta(x)\,d\mu,
         \int_\Theta\sup f_\theta(x)\,d\mu\right].
\end{equation}
For families of the same orientation and common domain interior $I$,
integration preserves the pointwise resolvent order, is additive,
and commutes with multiplication by positive scalars.
\end{proposition}
\begin{proof}
The endpoint formula is \eqref{eq:aumanninterval}.
For a nondecreasing family, write
$\ell_\theta(x)=\inf f_\theta(x)$ and
$u_\theta(x)=\sup f_\theta(x)$.
On $[a,b]\subset I$, every $v\in f_\theta(x)$ satisfies
\[
 |v|\leq |\ell_\theta(a)|+|u_\theta(b)|.
\]
This integrable bound allows dominated convergence in the one-sided
identities
$\ell_\theta(x)=\lim_{y\uparrow x}u_\theta(y)$ and
$u_\theta(x)=\lim_{y\downarrow x}\ell_\theta(y)$.
The integrated endpoints satisfy the same identities.
Definition~\ref{def:completedrestriction} therefore gives their unique
maximal completion with domain contained in $\overline I$.
Negation handles nonincreasing families. Endpoint comparison,
addition and positive scaling prove the remaining assertions on $I$.
The order extends by the separated-graph comparison in
Proposition~\ref{prop:comparison}. Sums and positive scalar multiples
are maximal with domain interior $I$, so uniqueness of completion
extends the two algebraic identities.
\end{proof}

Under the integrability assumptions of
Definition~\ref{def:hardytransform}, Hardy averaging is the
completed family integral of the dilations
$\left.(f\circ(r\id))\right|_I$, $0<r<1$.
For $I=(0,L)$ and $p\in I$, ordinary substitution gives its scalar
value $p^{-1}\int_0^p f(u)\,du$.

\section{Hardy averaging and universal tail envelopes}\label{sec:hardy}
We represent the classical Hardy average \cite{HardyLittlewoodPolya1952}
by the operator integral from Subsection~\ref{sec:integration}.
Its scalar formula is given immediately below. For a law $\mu$ and
$X\sim\mu$, write $\SSS[\mu]=\SSS[X]$ and $\TT[\mu]=\TT[X]$.

\subsection{Definition}\label{subsec:hardydefinition}

\begin{definition}[Hardy transform]\label{def:hardytransform}
Let $f\in\Md$ be integrable and satisfy
$0=\inf\dom f<\sup\dom f$.
Its \emph{Hardy transform} is\footnote{Using Definitions~\ref{def:area}
and~\ref{def:familyintegral}, with Lebesgue measure on $(0,1)$ for the operator integral.}
\begin{equation}\label{eq:hardyfamily}
 \mathcal H(f)=\int_{(0,1)}
 \left.(f\circ(r\id))\right|_{\operatorname{int}\dom f}\,dr.
\end{equation}
\end{definition}

By \eqref{eq:integralvalues} and the substitution $u=rp$,
\begin{equation}\label{eq:hardy}
 \forall p\in\operatorname{int}\dom f:\qquad
 \mathcal H(f)(p)=\left\{\frac1p\int_0^p f\right\}.
\end{equation}

\begin{proof}[Proof of well-definedness in Definition~\ref{def:hardytransform}]
Put $I=(0,L)=\operatorname{int}\dom f$ and choose a
measurable selection $s$ of $f$ on $I$. For $p\in I$,
ordinary substitution gives
\[
 \int_0^1 |s(rp)|\,dr
 =\frac1p\int_0^p |s(u)|\,du<\infty.
\]
The endpoint maps $r\mapsto\inf f(rp)$ and
$r\mapsto\sup f(rp)$ are measurable and agree with $s(rp)$
except at countably many arguments. Both are therefore
integrable. The completed dilations restricted to $I$ have
common domain interior $I$, so
Proposition~\ref{thm:familyintegral} proves that their integral
is well defined and has singleton values on $I$.
\end{proof}

The same notation extends to $f\in\Md$ with domain interior
$(0,L)$ that is integrable on every $(0,b)$, $0<b<L$.
For $0<p<b<L$, define $\mathcal H(f)(p)$ using the Hardy transform
of the completed restriction $f|_{(0,b)}$.
Formula~\eqref{eq:hardy} makes this independent of $b$; maximal
completion then defines the full operator. The order, algebraic,
convexity and concavity properties below extend to this setting by
applying them to these integrable restrictions.

\begin{proposition}[Order and algebraic properties of the Hardy transform]\label{prop:hardycalculus}
Let $f,g\in\Md$ be integrable, with common domain interior
$I=(0,L)$ for some $L\in(0,\infty]$. Then, for every $\lambda>0$,
\[
 \mathcal H(f+g)=\mathcal H(f)+\mathcal H(g),\qquad \mathcal H(\lambda f)=\lambda \mathcal H(f).
\]
Moreover, $f\leq \mathcal H(f)$, and $f\leq g$ implies $\mathcal H(f)\leq \mathcal H(g)$.
The latter assertion also holds for integrable $f,g\in\Md$ with
respective domain interiors $(0,L_f)$ and $(0,L_g)$,
where $L_f,L_g\in(0,\infty]$.
\end{proposition}

\begin{proof}
Completed restriction and positive dilation commute with sums and
positive scalar multiplication. Proposition~\ref{thm:familyintegral}
therefore gives the two algebraic identities directly.
For $r\in(0,1)$, monotonicity gives
$f\leq\left.(f\circ(r\id))\right|_I$; integrating this comparison
gives $f\leq \mathcal H(f)$.

For $f\leq g$ on possibly different averaging intervals, choose
selections $s_f,s_g$. The separated-graph characterization gives
\[
 \forall p\in(0,L_f),\ \forall q\in(0,L_g)\text{ with }p>q,
 \ \forall r\in(0,1):\qquad s_f(rp)\leq s_g(rq).
\]
Integrating over $(0,1)$ and using the Aumann representation in
Proposition~\ref{thm:familyintegral} proves the separated comparison for
the interior values of $\mathcal H(f),\mathcal H(g)$.
Proposition~\ref{prop:orderdomains} gives $L_f\leq L_g$; one-sided
limits include the downward boundary ray of $\mathcal H(f)$ at $L_f$ and the
upward boundary ray of $\mathcal H(g)$ at zero, whenever they exist.
Proposition~\ref{prop:comparison} yields $\mathcal H(f)\leq \mathcal H(g)$.
\end{proof}

A maximal monotone operator with nonempty domain interior is
called convex or concave if it admits a convex or concave
selection there. Such a selection is continuous, so the
operator is single-valued on its domain interior.

\begin{lemma}[Convexity and concavity of the Hardy transform]\label{lem:hardyconvexity}
Let $f\in\Md$ be integrable and satisfy
$0=\inf\dom f<\sup\dom f$.
If $f$ is convex (respectively, concave), then $\mathcal H(f)$ is convex
(respectively, concave).
\end{lemma}
This follows because dilation and integration preserve the defining
convexity or concavity inequality.

For integrable random variables, the Hardy transform admits the
classical expected-shortfall representations
\cite{ArtznerDelbaenEberHeath1999,AcerbiTasche2002,RockafellarUryasev2002CVaR}.
The green regions in Figure~\ref{fig:shortfallarea} represent the same
excess area, computed in quantile and survival coordinates.

\begin{figure}[H]
 \centering
 \includegraphics[width=\textwidth]{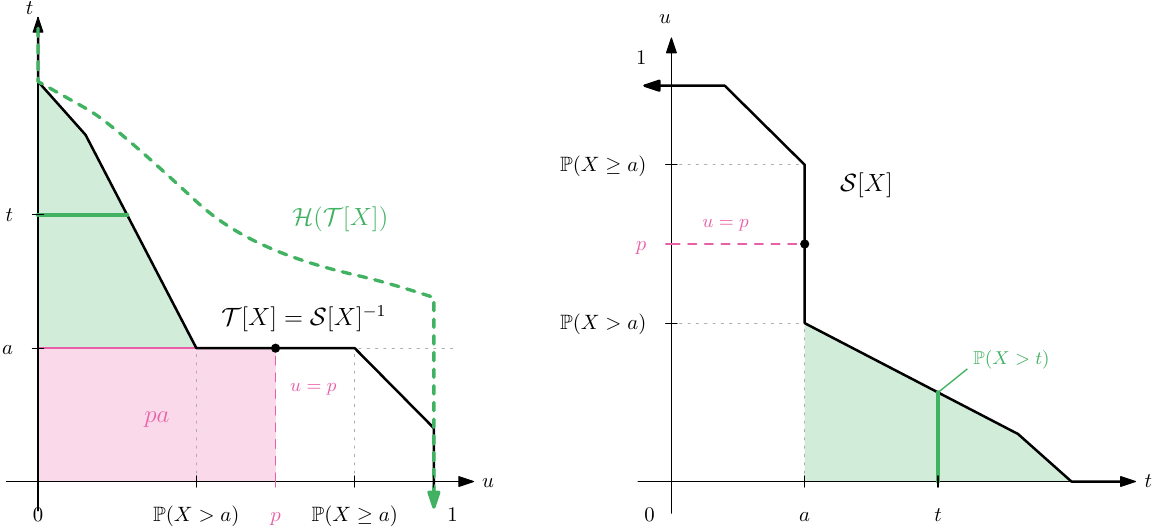}
 \caption{Area identity for $0<p<1$ and $a\in\TT[X](p)$, illustrated
 with $a>0$:
 $\int_0^p\TT[X](u)\,du-pa=\int_a^\infty\PP(X>t)\,dt=\E\left[(X-a)_+\right]$.}
 \label{fig:shortfallarea}
\end{figure}

The minimization formula and its minimizer interval are classical
\cite[Theorem~10]{RockafellarUryasev2002CVaR}.
\begin{lemma}[Variational formulas for the Hardy transform]\label{lem:hardyvariational}
For $X\in L^1$ and $0<p<1$,
\begin{align}\label{eq:hardyvariational}
 \mathcal H(\TT[X])(p)
 &=\min_{a\in\R}\left\{a+\frac1p\E\left[(X-a)_+\right]\right\}
 =\frac1p\sup_{\substack{0\leq Z\leq1\\\E Z=p}}\E[XZ],
\end{align}
where $Z$ ranges over random variables on the original probability space. The value of the completed quantile operator at $p$ is the minimizer set:
\[
 \TT[X](p)=\operatorname*{argmin}_{a\in\R}
 \left\{a+\frac1p\E\left[(X-a)_+\right]\right\}.
\]
\end{lemma}

\begin{proof}
Fix $0<p<1$. The finite convex objective
$\Phi_p(a):=a+p^{-1}\E\left[(X-a)_+\right]$ has one-sided derivatives
\[
 \Phi'_{p,-}(a)=1-p^{-1}\PP(X\ge a),\qquad
 \Phi'_{p,+}(a)=1-p^{-1}\PP(X>a).
\]
Here $\partial\Phi_p$ denotes the ordinary convex subdifferential,
whose interval endpoints are the left and right derivatives:
\[
 \partial\Phi_p(a)
 =[\Phi'_{p,-}(a),\Phi'_{p,+}(a)]
 =1-p^{-1}\SSS[X](a).
\]
Thus
\[
 0\in\partial\Phi_p(a)
 \quad\Longleftrightarrow\quad p\in\SSS[X](a)
 \quad\Longleftrightarrow\quad a\in\TT[X](p).
\]
Since $\TT[X](p)$ is nonempty, $\Phi_p$ attains its minimum exactly
on this set, including when the minimizer is nonunique.

Choose a minimizer $a\in\TT[X](p)$ and a decreasing quantile
selection $q$. Then, up to endpoint choices, for $u<p$ we have
$q(u)\geq a$, and for $u>p$ we have $q(u)\leq a$.
The quantile representation gives
\[
 \int_0^p(q(u)-a)\,du
 =\int_0^1(q(u)-a)_+\,du
 =\E\left[(X-a)_+\right].
\]
Dividing by $p$ and using~\eqref{eq:hardy} yields
$\Phi_p(a)=\mathcal H(\TT[X])(p)$, proving the first equality
in~\eqref{eq:hardyvariational}.

For every feasible $Z$, the pointwise inequality
$(X-a)Z\le(X-a)_+$ gives $\E[XZ]\le pa+\E\left[(X-a)_+\right]$.
At a minimizing $a$, $\PP(X>a)\leq p\leq\PP(X\geq a)$ allows choosing
$0\le\eta\le1$ with $\eta\PP(X=a)=p-\PP(X>a)$.
Then $Z=\mathbf1_{\{X>a\}}+\eta\mathbf1_{\{X=a\}}$ satisfies $\E Z=p$
and $(X-a)Z=(X-a)_+$, so it attains equality and proves the dual formula.
\end{proof}

\subsection{Aggregation of Hardy tail profiles}
\label{subsec:hardyaggregation}

We now ask for the largest upper-tail average of an aggregate when the
individual upper-tail averages are prescribed. The relevant scalar
majorant is the superadditive transform defined below.

Throughout this subsection, let $F:[0,\infty)^n\to[0,\infty)$ be
continuous, coordinatewise nondecreasing, and positively homogeneous.
Its superadditive transformation is
\begin{equation}\label{eq:superadditivetransform}
 F^\sharp(a)
 :=
 \sup\left\{
 \sum_{j=1}^mF(v^{(j)}),
 m\ge1,\quad \forall j\in[m]:
 v^{(j)}\in[0,\infty)^n,\;
 \sum_{j=1}^mv^{(j)}\leq a
 \right\}.
\end{equation}
The constraint is coordinatewise. This is the superadditive transform of
\cite[Definition~1]{GrecoMesiarRindoneSipeky2016}, and is the least
superadditive majorant of $F$. Write $\mathbbm{1}:=(1,\ldots,1)$.
The familiar bounds are
\[
 F(a)\le F^\sharp(a)
 \le F(\mathbbm{1})\sum_{i=1}^{n}a_i.
\]
Thus $F^\sharp$ is finite. It is also coordinatewise
nondecreasing. Scaling and concatenating admissible families show that it is
positively homogeneous and superadditive, hence concave. It is therefore
also the least positively homogeneous concave majorant of $F$, as in
\cite[Remark~4]{GrecoMesiarRindoneSipeky2016}.

\Needspace{14\baselineskip}
\begin{example}[Superadditive transformation of an $\ell^q$-norm]
\label{ex:normtransform}
Let $n\ge1$ and $1\le q\le\infty$. For $x\in[0,\infty)^n$, set
$F(x):=\|x\|_q$. Then
\begin{equation}\label{eq:normtransform}
 \forall a\in[0,\infty)^n:\qquad
 F^\sharp(a)=\|a\|_1=\sum_{i=1}^{n}a_i.
\end{equation}
For every admissible family in \eqref{eq:superadditivetransform},
\[
 \sum_j\|v^{(j)}\|_q
 \le\sum_j\|v^{(j)}\|_1
 =\left\|\sum_jv^{(j)}\right\|_1
 \le\|a\|_1.
\]
Conversely, the family $v^{(i)}=a_i e_i$, where $e_i$ is the $i$th
coordinate vector, is admissible and satisfies
$\sum_i\|a_i e_i\|_q=\sum_i a_i$. This proves
\eqref{eq:normtransform}. For example, for the Euclidean norm,
$F(\mathbbm{1})=\sqrt n$ while $F^\sharp(\mathbbm{1})=n$.

Consequently, Proposition~\ref{prop:hardytailaggregation} below gives,
for nonnegative integrable random variables $X_1,\ldots,X_n$ and
$0<p<1$,
\begin{equation}\label{eq:hardytailnorm}
 \mathcal H\left(\TT[\|(X_1,\ldots,X_n)\|_q]\right)
 \le\sum_{i=1}^n\mathcal H(\TT[X_i]).
\end{equation}
Theorem~\ref{thm:hardytailaggregationsharp} shows that this bound is
sharp, even for finite-valued comonotone variables.
Thus, when only the individual upper-tail averages are prescribed,
the best bound for the upper-tail average of the norm is their sum,
for every $q\in[1,\infty]$. The extremizing construction uses large
coordinate values at successively rarer probability scales.

For $0<q\le1$, the function
$F(x):=\|x\|_q=(\sum_{i=1}^n x_i^q)^{1/q}$ (a quasi-norm for $q<1$)
is superadditive on $[0,\infty)^n$, so
$F^\sharp(a)=F(a)=\|a\|_q$.
\end{example}

As a second illustration, we compute the superadditive transformation of
an odd median. The ternary case is treated in
\cite[equations~(4)--(5)]{GrecoMesiarRindoneSipeky2016}.
For arbitrary odd arity, the scalar formula~\eqref{eq:mediantransform}
has a classical probabilistic antecedent: the sharp
R\"uger--Morgenstern bound \cite{Ruger1978,Morgenstern1980}
for the probability that at least $r+1$ of $2r+1$ events occur is its
right-hand side, capped at $1$, with the event probabilities in place
of $a_i$. Choosing probabilities $a_i/C$ with $C>\sum_i a_i$ removes
the cap after scaling by $C$ and recovers the scalar formula by
homogeneity. For completeness, we give a proof in
Section~\ref{app:mediantransform}. The further conclusion in
Corollary~\ref{cor:hardymedian} is sharpness for prescribed marginal
upper-tail means, even for finite-valued comonotone variables.

\begin{proposition}[Superadditive transformation of an odd median \cite{Ruger1978,Morgenstern1980}]
\label{prop:mediantransform}
Let $r\ge0$ be an integer and set
\[
 n=2r+1,
 \qquad
 M(x):=\operatorname{med}(x_1,\ldots,x_n).
\]
For $a\in[0,\infty)^n$, write
\[
 a_{(1)}\le\cdots\le a_{(n)}
\]
for its ordered coordinates. Then
\begin{equation}\label{eq:mediantransform}
 M^\sharp(a)
 =
 \min_{1\le\ell\le r+1}
 \frac{a_{(1)}+\cdots+a_{(r+\ell)}}{\ell}.
\end{equation}
\end{proposition}

For $f_1,\ldots,f_n\in\Md$ with nonnegative values on a common
domain interior $I=(0,L)$, write $F(f_1,\ldots,f_n)$ for the maximal
decreasing completion of the operator with interval values
\[
 \forall u\in I:\qquad
 \left[
 F\bigl(\inf f_1(u),\ldots,\inf f_n(u)\bigr),\,
 F\bigl(\sup f_1(u),\ldots,\sup f_n(u)\bigr)
 \right].
\]
Continuity and monotonicity of $F$ preserve the one-sided limits, so
these interval values determine an operator in $\Md$. For differing domains,
the combination is formed on their common interior and then maximally
completed. We identify nonincreasing step functions with their maximal
completions as in Definition~\ref{def:completedrestriction}.
Their jump values do not affect their Lebesgue integrals.

\begin{lemma}[Area inequality for monotone aggregation]
\label{lem:aggregationarea}
Let $f_1,\ldots,f_n\in\Md$ have nonnegative values on their common
domain interior $(0,L)$. If $0<b\le L$ is finite and, for every
$i\in[n]$, $\int_0^b f_i<\infty$, then
\begin{equation}\label{eq:aggregationarea}
 \int_0^b F(f_1,\ldots,f_n)
 \le
 F^\sharp\left(\int_0^b f_1,\ldots,\int_0^b f_n\right).
\end{equation}
\end{lemma}
\begin{proof}
Choose nonnegative scalar selections $q_i$ of $f_i$ and put
$a_i=\int_0^b q_i(u)\,du$. If $a_i=0$, then $q_i=0$ almost
everywhere; remove these coordinates before applying Jensen on the
remaining coordinate face. The finite concave function $F^\sharp$
is continuous on the relative interior of that face. If all means
vanish, the bound is immediate. Otherwise concave Jensen under
normalized Lebesgue measure and homogeneity give
\[
 \int_0^b F(q_1(u),\ldots,q_n(u))\,du
 \leq\int_0^b F^\sharp(q_1(u),\ldots,q_n(u))\,du
 \leq bF^\sharp(a/b)=F^\sharp(a).
\]
The scalar selection convention identifies these integrals with
the areas in~\eqref{eq:aggregationarea}.
\end{proof}

In particular, for every $p\in I$, divide
\eqref{eq:aggregationarea}, with $b=p$, by $p$ and use homogeneity
of $F^\sharp$ and~\eqref{eq:hardy} to obtain
\begin{equation}\label{eq:hardyaggregation}
 \mathcal H\bigl(F(f_1,\ldots,f_n)\bigr)(p)
 \le
 F^\sharp\bigl(
 \mathcal H(f_1)(p),\ldots,\mathcal H(f_n)(p)
 \bigr).
\end{equation}

For nonnegative integrable random variables, the relevant decreasing
operators are $\TT[X_i]=\SSS[X_i]^{-1}$, with common domain interior
$(0,1)$. The bound $F(X_1,\ldots,X_n)\le F(\mathbbm{1})\sum_iX_i$
ensures that the aggregate is also integrable.

\begin{proposition}[Aggregation of Hardy tail profiles]
\label{prop:hardytailaggregation}
Let $X_1,\ldots,X_n$ be nonnegative integrable random variables on
one probability space. For every $p\in(0,1)$,
\begin{equation}\label{eq:hardytailaggregation}
 \mathcal H\left(\TT[F(X_1,\ldots,X_n)]\right)(p)
 \le
 F^\sharp\left(
 \mathcal H(\TT[X_1])(p),\ldots,
 \mathcal H(\TT[X_n])(p)
 \right).
\end{equation}
\end{proposition}

\begin{proof}
For any nonnegative integrable $Y$, Lemma~\ref{lem:hardyvariational}
gives
\begin{equation}\label{eq:hardytailvariational}
 \mathcal H(\TT[Y])(p)
 =\frac1p\sup_{\substack{0\le\xi\le1\\\E\xi=p}}\E[\xi Y].
\end{equation}
For each feasible $\xi$, apply concave Jensen to $F^\sharp$ under
the probability measure $\xi\,d\PP/p$. Its coordinatewise
monotonicity and \eqref{eq:hardytailvariational} give
\[
 \begin{aligned}
 \frac1p\E[\xi F(X_1,\ldots,X_n)]
 &\le\frac1p\E[\xi F^\sharp(X_1,\ldots,X_n)]
 \le F^\sharp\left(
      \frac1p\E[\xi X_1],\ldots,\frac1p\E[\xi X_n]\right)\\
 &\le F^\sharp\left(
      \mathcal H(\TT[X_1])(p),\ldots,
      \mathcal H(\TT[X_n])(p)\right).
 \end{aligned}
\]
Zero weighted means force the corresponding coordinates to vanish
under this measure, so Jensen is applied on the remaining coordinate
face. Take the supremum over $\xi$ to conclude.
\end{proof}

For $\beta\in\Md$ whose inverse has a Hardy transform, write
\begin{equation}\label{eq:hardytailtransform}
 \mathfrak H(\beta):=\mathcal H(\beta^{-1})^{-1}.
\end{equation}
Explicit power and exponential profiles are given in
\eqref{eq:hardyregularizedprofiles} below.
For an integrable law $\mu$, $\mathfrak H(\SSS[\mu])$ is the completed
survival representation of its Hardy--Littlewood transform, equivalently
of the superdistribution obtained by inverting its superquantile
\cite[Section~4]{RockafellarRoyset2014}.

Suppose that, for each $i\in\{1,\ldots,n\}$, the operator
$\beta_i\in\Md$ satisfies
\[
 \SSS[X_i]\leq\beta_i\leq1,\qquad
 \inf\ran\beta_i=0,\qquad
 \int_0^1|\beta_i^{-1}(p)|\,dp<\infty.
\]
The integral uses the scalar convention above. These assumptions make
$\mathcal H(\beta_i^{-1})$ well defined on $(0,1)$.
Inversion and Hardy order preservation give $\mathcal H(\TT[X_i])\le\mathcal H(\beta_i^{-1})$;
in particular, the latter profiles are nonnegative on $(0,1)$.
Combining this with \eqref{eq:hardytailaggregation} and
$\TT[F(X_1,\ldots,X_n)]\le\mathcal H(\TT[F(X_1,\ldots,X_n)])$
and then inverting yields
\begin{equation}\label{eq:tailaggregation}
 \SSS[F(X_1,\ldots,X_n)]
 \le
 \left[F^\sharp\left(
 \mathcal H(\beta_1^{-1}),\ldots,
 \mathcal H(\beta_n^{-1})\right)\right]^{-1}
\end{equation}
If $F\equiv0$, the aggregate is identically zero and its survival
bound is immediate. Otherwise $F^\sharp(\mathbbm{1})>0$: indeed,
$F(a)>0$ for some $a\geq0$, and monotonicity and homogeneity imply
$F(\mathbbm{1})>0$. For a common admissible envelope $\beta$, homogeneity gives
$\SSS[F(X_1,\ldots,X_n)]\le\mathfrak H(\beta)\circ(\id/F^\sharp(\mathbbm{1}))$;
equivalently, Lemma~\ref{lem:survivalcomparison} gives
\begin{equation}\label{eq:commontailaggregation}
 \begin{aligned}
 &\forall t\in F^\sharp(\mathbbm{1})\cdot\dom\mathfrak H(\beta):
 \qquad
 \PP\bigl(F(X_1,\ldots,X_n)>t\bigr)
 \le\inf\mathfrak H(\beta)\left(\frac{t}{F^\sharp(\mathbbm{1})}\right).
 \end{aligned}
\end{equation}
When $F$ is the sum,~\eqref{eq:tailaggregation} rewrites
\begin{equation}\label{eq:hardytailsum}
 \SSS\!\left[\sum_{i=1}^n X_i\right]
 \leq
 \left[\sum_{i=1}^n\mathcal H(\TT[X_i])\right]^{-1}.
\end{equation}

To make the common-envelope bound \eqref{eq:commontailaggregation}
and the sum bound \eqref{eq:hardytailsum} explicit, we now compute
$\mathfrak H$ for power and exponential profiles.
For $\kappa>1$ and $t>0$, set $P_\kappa(t)=\{t^{-\kappa}\}$, and set $P_\kappa(t)=\varnothing$ elsewhere. For $t\in\R$, set $\mathcal E(t)=\{e^{-t}\}$. Both belong to $\Md$. Their inverse operators have averaging interval $(0,\infty)$, and scalar area integration gives
\begin{equation}\label{eq:hardyprofileidentities}
 \forall p>0:\qquad
 \mathcal H(P_\kappa^{-1})(p)=\frac\kappa{\kappa-1}p^{-1/\kappa},\qquad
 \mathcal H(\mathcal E^{-1})(p)=1-\log p.
\end{equation}
Consequently, for $C>0$,
\begin{equation}\label{eq:hardyregularizedprofiles}
 \mathfrak H(CP_\kappa)
 =C\left(\frac\kappa{\kappa-1}\right)^\kappa P_\kappa,
 \qquad \mathfrak H(C\mathcal E)=Ce\,\mathcal E.
\end{equation}
The constants in \eqref{eq:hardyregularizedprofiles} are familiar from
VaR/ES comparisons: Li and Wang
\cite[Example~5]{LiWang2023PELVE} compute the same Pareto and exponential
constants, including the limit as $\kappa\to\infty$.
For a continuous integrable law, our tail mass $p$ corresponds to confidence
$1-p$: writing $q(p)=\operatorname{VaR}_{1-p}(X)$ gives
$\mathcal H(q)(p)=\operatorname{ES}_{1-p}(X)$.
Their Remark~2 compares this mean with $q(p/c)$ at a rescaled tail mass.
Their Pareto survival is capped at $1$, whereas $P_\kappa$ here is an
uncapped profile on $(0,\infty)$; the full operator identities also
record domains and maximal completion.
Inversion and Hardy order preservation extend these bounds to any
admissible envelope dominated by the corresponding profile. More
generally, the same constants follow from convexity of a scalar
transformation, as stated next. The increasing-concave transformation
comparison in \cite[Theorem~1(ii)]{LiWang2023PELVE} uses the same Jensen
mechanism in the distributional setting. We include the proofs of the
next two propositions in Sections~\ref{app:hardyprofiles}
and~\ref{app:convextransforms}, with the
integrability and domain details required for these profiles.

\begin{proposition}[Hardy bounds from convex transformations]\label{prop:hardyprofiles}
Let $\alpha\in\Md$ be single-valued, continuous and strictly decreasing
on its open interval domain $D$, mapping $D$ onto $(0,L)$ with
$0<L\leq\infty$.
If, for some $\kappa>1$, $t\mapsto\alpha(t)^{-1/\kappa}$ is convex on $D$, then $\mathfrak H(\alpha)$ is defined and
\begin{equation}\label{eq:hardypowerbound}
 \mathfrak H(\alpha)
 \leq\left(\frac\kappa{\kappa-1}\right)^\kappa\alpha.
\end{equation}
If instead $t\mapsto-\log\alpha(t)$ is convex on $D$, then $\mathfrak H(\alpha)$ is defined and
\begin{equation}\label{eq:hardyexponentialbound}
 \mathfrak H(\alpha)\leq e\alpha.
\end{equation}
The constants are attained for $\alpha=CP_\kappa$ and
$\alpha=C\mathcal E$, respectively, with $C>0$.
\end{proposition}

The factor $\left(\frac\kappa{\kappa-1}\right)^\kappa$ in
\eqref{eq:hardypowerbound} converges to the factor $e$ in
\eqref{eq:hardyexponentialbound} as $\kappa\to\infty$. The following
scalar statement gives a corresponding connection between power and
exponential changes of variables. For background on convex
transformations, see \cite{vanZwet1964}.
\begin{proposition}[Power and exponential convex transformations]\label{prop:convextransforms}
Let $\varphi:(0,\infty)\to\R$ be continuous. If, for every $\kappa>0$, $x\mapsto\varphi(x^{-\kappa})$ is convex on $(0,\infty)$, then $x\mapsto\varphi(e^{-x})$ is convex on $\R$. The converse holds when $\varphi$ is nondecreasing. The corresponding concave statements hold, with $\varphi$ nonincreasing for the converse.
\end{proposition}

\subsection{Sharpness results}
We call $X_1,\ldots,X_n$ \emph{comonotone} when they admit a joint
representation
\[
 \forall i\in[n]:\qquad X_i=q_i(U),
\]
with $U$ uniform on $(0,1)$ and all $q_i$ nonincreasing. Their
tail quantile operators are the maximal completions of these profiles.
Consequently, monotone aggregation commutes with their completed
quantile representation:
\[
 \TT[F(X_1,\ldots,X_n)]
 =F(\TT[X_1],\ldots,\TT[X_n]).
\]

\begin{theorem}[Sharpness of Hardy tail aggregation]
\label{thm:hardytailaggregationsharp}
Let $F:[0,\infty)^n\to[0,\infty)$ be continuous, coordinatewise
nondecreasing, and positively homogeneous.
For every $p\in(0,1)$ and $a\in[0,\infty)^n$,
\begin{equation}\label{eq:hardytailaggregationsharp}
 \begin{aligned}
 F^\sharp(a)=\sup\Bigl\{&
 \mathcal H\left(\TT[F(X_1,\ldots,X_n)]\right)(p):\quad
 \forall i\in[n]:\ \mathcal H(\TT[X_i])(p)=a_i,\\
 &X_1,\ldots,X_n\text{ nonnegative and finite-valued}
 \Bigr\}.
 \end{aligned}
\end{equation}
The same supremum is obtained by restricting to comonotone variables.
\end{theorem}
Here $p$ and the individual upper-tail averages are fixed, but the
marginal laws vary. The supremum
need not be attained simultaneously for every $p$.

For comparison, for
$0<p<1$ and $a\in[0,\infty)^n$,
\begin{equation}\label{eq:comonotonequantilesharp}
 \begin{aligned}
 F(a)=\sup\Bigl\{&
 b\in\TT[F(X_1,\ldots,X_n)](p):\quad
 \forall i\in[n]:\ \TT[X_i](p)=\{a_i\},\\
 &X_1,\ldots,X_n\text{ nonnegative, finite-valued, and comonotone}
 \Bigr\}.
 \end{aligned}
\end{equation}
Unlike the one of Theorem~\ref{thm:hardytailaggregationsharp}, that identity generally fails under arbitrary (non comonotone) dependence.

For every admissible tuple, the aggregate quantile operator takes the
singleton value $\{F(a)\}$ at $p$, and the constant variables
$X_i\equiv a_i$ attain this value.
\begin{proof}[Proof of Theorem~\ref{thm:hardytailaggregationsharp}]
Proposition~\ref{prop:hardytailaggregation} gives the upper bound.
For the reverse bound, fix a finite family
$v^{(j)}\in[0,\infty)^n$ with $c:=\sum_{j=1}^mv^{(j)}\le a$
and $0<\varepsilon<1$. Set
\[
 \forall j\in[m]:\qquad s_j:=p\varepsilon^{j-1},\qquad s_{m+1}:=0.
\]
Let $U$ be uniform on $(0,1)$ and, for $0<u<1$, define
\[
 q_i^\varepsilon(u):=
 a_i-c_i+\sum_{j=1}^m\frac{p\,v_i^{(j)}}{s_j}\mathbf1_{(0,s_j)}(u),
 \qquad X_i^\varepsilon:=q_i^\varepsilon(U).
\]
These variables are nonnegative, finite-valued, and comonotone.
Their quantile operators $\TT[X_i^\varepsilon]$ are the maximal
completions of the decreasing step profiles $q_i^\varepsilon$ on $(0,1)$.
Equation~\eqref{eq:hardy} and scalar integration give
\[
 \mathcal H(\TT[X_i^\varepsilon])(p)
 =\frac1p\int_0^p\TT[X_i^\varepsilon](u)\,du
 =a_i-c_i+\sum_{j=1}^mv_i^{(j)}=a_i.
\]
Write $q^\varepsilon=(q_1^\varepsilon,\ldots,q_n^\varepsilon)$.
On $(s_{j+1},s_j)$ one has
$q^\varepsilon(u)\ge p\,v^{(j)}/s_j$.
The completed quantile representation for comonotone variables,
monotonicity, and homogeneity therefore give
\[
 \begin{aligned}
 \mathcal H\left(
   \TT[F(X_1^\varepsilon,\ldots,X_n^\varepsilon)]\right)(p)
&=\frac1p\int_0^pF(q^\varepsilon(u))\,du\\
 &\ge\frac1p\sum_{j=1}^m(s_j-s_{j+1})
                       F\left(\frac{p\,v^{(j)}}{s_j}\right)
\ge(1-\varepsilon)\sum_{j=1}^mF(v^{(j)}).
 \end{aligned}
\]
Let $\varepsilon\downarrow0$ and take the supremum over the admissible
families in \eqref{eq:superadditivetransform} to obtain $F^\sharp(a)$.
\end{proof}

Thus $F^\sharp$ is the smallest bound depending only on the individual
Hardy tail averages, even for finite-valued comonotone variables.

The lattice median acts on the marginal profiles; the median of a
random vector is a different object, whose
upper-tail averages are controlled below.

\begin{corollary}[Sharp Hardy bound for an odd median]
\label{cor:hardymedian}
Let $r\ge0$ be an integer, $n=2r+1$, and
$X_1,\ldots,X_n$ be nonnegative integrable random variables on
one probability space. Then
\begin{equation}\label{eq:hardytailmedian}
 \mathcal H\left(\TT[\operatorname{med}(X_1,\ldots,X_n)]\right)
 \le
 \inf_{1\le\ell\le r+1}
 \inf_{\substack{J\subseteq[n]\\|J|=r+\ell}}
 \frac1\ell\sum_{i\in J}\mathcal H(\TT[X_i]).
\end{equation}
This bound is sharp at every fixed $p\in(0,1)$, already for
nonnegative finite-valued comonotone variables.
\end{corollary}

\begin{proof}
For a scalar vector $a$, the sum of its $r+\ell$ smallest coordinates
is $\min_{|J|=r+\ell}\sum_{i\in J}a_i$. Consequently,
Propositions~\ref{prop:mediantransform} and~\ref{prop:hardytailaggregation}
give the corresponding scalar bound for every $p\in(0,1)$.
By Proposition~\ref{prop:finitevalues}, the right-hand side of
\eqref{eq:hardytailmedian} is single-valued at $p$ and takes this scalar
minimum as its value.
Maximal completion gives the stated operator comparison.
Theorem~\ref{thm:hardytailaggregationsharp}, applied to the scalar median,
proves sharpness.
\end{proof}

\begin{example}[Heterogeneous estimators: median and normalized sum]
\label{ex:heterogeneousestimators}
Let $\widehat\theta_1,\widehat\theta_2,\widehat\theta_3$ be integrable
scalar estimators of $\theta\in\R$. Set
\[
 \widehat\theta_{\mathrm{med}}
 :=\operatorname{med}(\widehat\theta_1,\widehat\theta_2,\widehat\theta_3),
 \qquad
 \widehat\theta_{\mathrm{av}}:=\frac13\sum_{i=1}^3\widehat\theta_i.
\]
Fix $p\in(0,1)$ and write
\[
 \forall i\in[3]:\qquad
 a_i:=\mathcal H\bigl(\TT[|\widehat\theta_i-\theta|]\bigr)(p).
\]
Let $a_{(1)}\leq a_{(2)}\leq a_{(3)}$ be these upper-tail means
in increasing order. Then
\begin{equation}\label{eq:heterogeneousestimators}
 \begin{aligned}
 \mathcal H\bigl(\TT[|\widehat\theta_{\mathrm{med}}-\theta|]\bigr)(p)
 &\leq\min\left\{a_{(1)}+a_{(2)},\frac{a_1+a_2+a_3}{2}\right\},\\
 \mathcal H\bigl(\TT[|\widehat\theta_{\mathrm{av}}-\theta|]\bigr)(p)
 &\leq\frac{a_1+a_2+a_3}{3}.
 \end{aligned}
\end{equation}
For the first bound, apply Corollary~\ref{cor:hardymedian} to the
absolute errors and use
\[
 |\widehat\theta_{\mathrm{med}}-\theta|
 \leq\operatorname{med}(|\widehat\theta_1-\theta|,
                       |\widehat\theta_2-\theta|,
                       |\widehat\theta_3-\theta|).
\]
The second bound follows from the triangle inequality and
Proposition~\ref{prop:hardytailaggregation} with
$F(x)=(x_1+x_2+x_3)/3$.
The median bound can omit the largest upper-tail mean, whereas the
bound for the normalized sum includes all three.

For a concrete illustration, let $Z\geq0$ satisfy $0<\E Z<\infty$
and let $b>6$. Take
\[
 \widehat\theta_1=\theta+Z,\qquad
 \widehat\theta_2=\theta+2Z,\qquad
 \widehat\theta_3=\theta+bZ.
\]
Writing $h:=\mathcal H(\TT[Z])(p)$, the marginal upper-tail means
are $(h,2h,bh)$. The bounds in \eqref{eq:heterogeneousestimators}
are therefore $3h$ for the median and $(b+3)h/3$ for the normalized
sum. The latter is larger and grows with $b$.
This difference can also be seen in the actual errors, since
\[
 \widehat\theta_{\mathrm{med}}-\theta=2Z,\qquad
 \widehat\theta_{\mathrm{av}}-\theta=\frac{b+3}{3}Z.
\]
Thus a single estimator with an arbitrarily large error scale can
spoil the normalized sum while the median remains controlled.
Permuting the inputs leaves both formulas unchanged, so this
comparison does not require knowing the index of the poor estimator.
\end{example}

Common-marginal worst-case quantile asymptotics are established by Wang and Wang \cite[Corollary~3.7]{WangWang2015END}; their confidence level is $1-p$ in our tail-mass convention. Finite-type heterogeneous portfolios are studied in \cite[Theorem~4.2]{PuccettiWangWang2013}. The following formulation uses completed tail quantile operators and permits different integrable laws by allowing repeated copies of each marginal.

\begin{proposition}[Universal quantile envelope as a lattice supremum]\label{prop:hardyoptimality}
Let $\mu_1,\ldots,\mu_n$ be integrable probability laws.
\begin{equation}\label{eq:hardylatticesup}
\frac1n\sum_{i=1}^n\mathcal H(\TT[\mu_i])
 =\sup_{\substack{m\geq1\\X_{i,j}\sim\mu_i\\1\leq i\leq n,\ 1\leq j\leq m}}
   \TT\!\left[\frac1{nm}\sum_{i=1}^n\sum_{j=1}^mX_{i,j}\right],
\end{equation}
\end{proposition}
For $n=1$, inversion in $\Md$ for the pointwise resolvent order recovers the common-law survival envelope.
\begin{proof}
For the proof, common-interval Hardy additivity gives
\[
 h:=\mathcal H\!\left(\frac1n\sum_{i=1}^n\TT[\mu_i]\right)
 =\frac1n\sum_{i=1}^n\mathcal H(\TT[\mu_i]),
\]
with the averaged quantile operator completed on $(0,1)$.
The dual formula in Lemma~\ref{lem:hardyvariational} gives, for every
replicated average and $0<p<1$,
\[
 \mathcal H\!\left(\TT\!\left[
 \frac1{nm}\sum_{i=1}^n\sum_{j=1}^mX_{i,j}\right]\right)(p)
 \le\frac1{nm}\sum_{i=1}^n\sum_{j=1}^m
       \mathcal H(\TT[X_{i,j}])(p)=h(p).
\]
Since each quantile operator lies below its Hardy transform, maximal
completion bounds the tail quantile of every replicated average by $h$.
Theorem~\ref{thm:lattice} therefore supplies
the supremum
\begin{align*}
  g:=\sup_{\substack{m\geq1\\X_{i,j}\sim\mu_i\\1\leq i\leq n,\ 1\leq j\leq m}}
   \TT\!\left[\frac1{nm}\sum_{i=1}^n\sum_{j=1}^mX_{i,j}\right]\leq h.
\end{align*}
The member quantile operators and $h$ have domain interior $(0,1)$,
so Proposition~\ref{prop:orderdomains} gives
$(0,1)\subset\dom g\subset[0,1]$.

Fix $0<p<1$, measurable selections $q_i(u)\in\TT[\mu_i](u)$, independent uniforms $U_{i,j}$ on $(0,1)$, and an independent Bernoulli variable $B$ of parameter $p$. Define
\[
 X_{i,j}=
 \begin{cases}
 q_i(pU_{i,j}),&B=1,\\
 q_i(p+(1-p)U_{i,j}),&B=0.
 \end{cases}
\]
Each $X_{i,j}$ has law $\mu_i$. For every $i$, conditional on $B=1$,
the sequence in $j$ is independent, identically distributed and
integrable. Under the conditional probability $\PP(\,\cdot\mid B=1)$,
the strong laws of large numbers give
\begin{align*}
 Y_m:=\frac1{nm}\sum_{i=1}^n\sum_{j=1}^mX_{i,j}
 \xrightarrow[m\to\infty]{\mathrm{a.s.}}
 \frac1{n}\sum_{i=1}^n\int_0^1 q_i(pu)\,du = \frac1{np}\sum_{i=1}^n\int_0^p q_i(u)\,du=h(p).
\end{align*}

For $(s,t)\in\gra g$, since $\TT[Y_m]\leq g$, inversion and Lemma~\ref{lem:survivalcomparison} give $\PP(Y_m>t)\leq s$.
If $t<h(p)$, almost sure convergence on $\{B=1\}$ implies that
$Y_m>t$ eventually on this event. Thus
$\liminf_{m\to\infty}\mathbf1_{\{Y_m>t\}}\ge\mathbf1_{\{B=1\}}$
almost surely, and Fatou's lemma gives
\[
 s\geq \liminf_{m\to\infty}\PP(Y_m>t)
 \ge\E\left[\liminf_{m\to\infty}\mathbf1_{\{Y_m>t\}}\right]
 \ge\PP(B=1)=p.
\]
Taking the contrapositive gives
\[
 s<p\quad\Longrightarrow\quad t\geq h(p).
\]
These are the separated inequalities for every interior graph point $(p,h(p))$ of $h$.

At $p=1$, completion gives $h(1)=(-\infty,\overline m]$, where $\overline m= \frac1n\sum_i\int x\,d\mu_i(x)$. For every $(s,v)\in\gra g$ with $s<1$, choose any $p\in(s,1)$; since $h(p)\geq\overline m$, this gives $v\geq\overline m$. At $p=0$, no separated comparison is needed, since $\dom g\subset[0,1]$. Proposition~\ref{prop:comparison} now gives $h\leq g$, proving \eqref{eq:hardylatticesup}. 
\end{proof}
The replication parameter $m$ is essential to the equality. If $\mu_1$ is the symmetric Bernoulli law on $\{0,1\}$ and $\mu_2=\delta_0$, then, for $1/2<p<1$, every fixed-pair average has tail quantile $0$, whereas $\frac12\sum_{i=1}^2\mathcal H(\TT[\mu_i])(p)=1/(4p)>0$.

\section{Randomized Lipschitz control on tails}\label{sec:local}

We turn to a second application of the completed tail-profile calculus.
Geometric Hardy averages of random continuity controls combine with a
two-sided Lipschitz extension estimate to give tail bounds, and the
P\'olya--Knopp inequality gives moment bounds. The logarithmic controls
are combined using linearity of expectation.

\subsection{The geometric Hardy transform}

Throughout this section, $I=(0,1)$. Products and scalar
transformations of profiles are formed on $I$ and then maximally
completed, as in Definition~\ref{def:completedrestriction}; in
particular, outward endpoint rays are not inserted into logarithms.
Integrals use the selection convention of Definition~\ref{def:area}.
For $t>0$, write $\log_+t=\max\{\log t,0\}$, with $\log_+0=0$.

For example, if $f$ is the maximal decreasing completion of the
zero function on $(0,1)$, then $(\exp\circ f)(1)=(0,1]$ is not
closed, so $\exp\circ f$ is not maximal.

\begin{definition}[Geometric Hardy transform]
\label{def:geometrichardytransform}
Let $q\in\Md$ have domain interior $I=(0,1)$ and nonnegative
values on $I$, with
\[
 \int_0^1\log_+q(s)\,ds<\infty.
\]
If $q$ does not vanish identically on $I$, its
\emph{geometric Hardy transform} $\mathcal G(q)$ is the unique
maximally nonincreasing extension of
\[
 \exp\circ\mathcal H\bigl(\log\circ\sup\{q,0\}\bigr)
\]
having domain interior $I$ and nonnegative values on $I$.
If $q$ vanishes identically on $I$, set $\mathcal G(q)=q$.

Equivalently, $\mathcal G(q)$ is the maximal decreasing completion
on $\overline I$ of the classical \emph{geometric mean operator}
defined on $(0,1)$ by
\begin{equation}\label{eq:local-geometric-definition}
 u\longmapsto
 \exp\left(\frac1u\int_0^u\log q(s)\,ds\right),
\end{equation}
with $\log0=-\infty$ and $\exp(-\infty)=0$.
\end{definition}

\begin{proof}
Suppose that $q$ does not vanish identically on $I$.
Lemma~\ref{lem:logmaximal} makes
$\log\circ\sup\{q,0\}$ maximally nonincreasing.
Its domain interior is the initial interval where $q$ is positive.
On every smaller initial interval, $\log q$ is integrable:
its positive part is integrable by assumption, and its negative
part is bounded by monotonicity.
Thus its Hardy transform is well defined and maximally
nonincreasing.

Exponentiation preserves monotonicity. At the right endpoint
of the logarithmic domain, its exponential either takes the value
$(0,c]$ for some $c>0$, or tends to zero.
If this endpoint lies inside $I$, maximality and nonnegativity
therefore require adding zero to that value and continuing by
zero to the right. Completion at $0$ and $1$ gives the unique
extension specified above.
Equation~\eqref{eq:hardy} yields
\eqref{eq:local-geometric-definition}.
\end{proof}

Zeros of $q$ can produce a jump
of $\mathcal G(q)$ as seen in the next example.
\begin{figure}[H]
 \centering
 \includegraphics[width=0.86\textwidth]{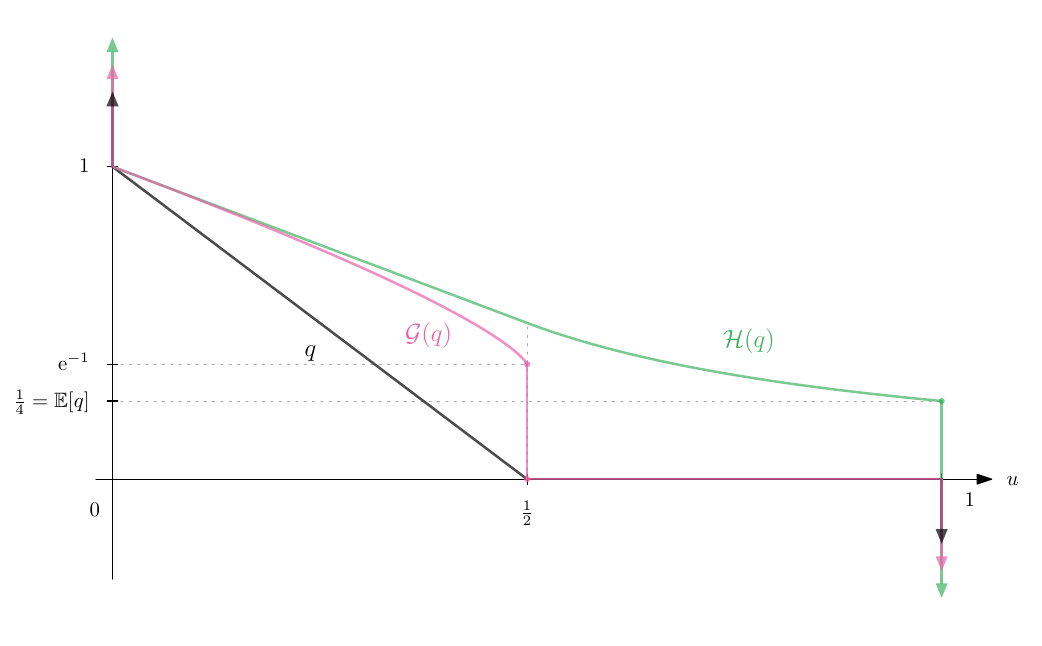}
 \caption[Geometric and Hardy averages near a zero]{%
 Example~\ref{ex:local-geometric-zero-jump}: $q$ (black),
 $\mathcal G(q)$ (pink), $\mathcal H(q)$ (green), maximally completed
 on $[0,1]$.
 \par{\centering$\displaystyle
 \begin{array}{c|ccc}
  u&q(u)&\mathcal G(q)(u)&\mathcal H(q)(u)\\ \hline
  (0,\tfrac12)&1-2u&\mathrm e^{-1}(1-2u)^{-(1-2u)/(2u)}&1-u\\
  (\tfrac12,1)&0&0&(4u)^{-1}
 \end{array}
 $\par}}
 \label{fig:geometric-zero-jump}
\end{figure}

\begin{example}[A zero-induced jump of the geometric average]
\label{ex:local-geometric-zero-jump}
Let $q$ be the maximal completion on $I$ of the continuous profile
\[
 \forall u\in(0,1):\qquad q(u)=\max\{1-2u,0\}.
\]
Although $q(1/2)=0$, the logarithm is integrable up to this level:
\[
 \int_0^{1/2}\log(1-2s)\,ds=-\frac12.
\]
Consequently,
\[
 \lim_{u\uparrow1/2}\mathcal G(q)(u)
 =\exp\left(2\int_0^{1/2}\log(1-2s)\,ds\right)
 =\mathrm e^{-1}.
\]
For every $u>1/2$, however, $q$ vanishes on the interval $(1/2,u)$,
so $\int_0^u\log q=-\infty$ and $\mathcal G(q)(u)=0$.
Maximal completion therefore fills the jump by setting
\[
 \mathcal G(q)(1/2)=[0,\mathrm e^{-1}].
\]
Here $\log q$ is finite on $(0,1/2)$, and $1/2$ is the upper
endpoint of the domain of its Hardy transform. The jump of
$\mathcal G(q)$ occurs at this boundary, which lies inside $I$.
Figure~\ref{fig:geometric-zero-jump} compares $q$, $\mathcal G(q)$
and the continuous Hardy average $\mathcal H(q)$.
\end{example}

The ordering $q\leq\mathcal G(q)\leq\mathcal H(q)$ visible in
Figure~\ref{fig:geometric-zero-jump} holds for every profile $q$
for which these transforms are defined, as the following lemma shows.

\begin{lemma}\label{lem:hardy-geometric-comparison}
Let $q$ satisfy Definition~\ref{def:geometrichardytransform}.
The transform $\mathcal G$ preserves order on this class, and
$q\leq\mathcal G(q)$. If, moreover, $q\in L^1(I)$, then
\begin{equation}\label{eq:hardy-geometric-comparison}
 q\leq \mathcal G(q)\leq\mathcal H(q).
\end{equation}
\end{lemma}
\begin{proof}
First suppose the profiles are strictly positive, and use decreasing
scalar selections on $I$. For $0<s<u$, monotonicity gives
$q(s)\geq q(u)$, so
\[
 \frac1u\int_0^u\log q(s)\,ds
 \geq\frac1u\int_0^u\log q(u)\,ds
 =\log q(u).
\]
Exponentiating gives $\mathcal G(q)(u)\geq q(u)$.
If $q$ is integrable, Jensen's inequality for the concave logarithm
also gives
\[
 \frac1u\int_0^u\log q(s)\,ds
 \leq\log\left(\frac1u\int_0^u q(s)\,ds\right).
\]
Exponentiating now yields $\mathcal G(q)\leq\mathcal H(q)$, the
usual comparison between geometric and arithmetic means.
Order preservation follows by integrating the inequality between
the logarithms of two ordered profiles.
Passing to decreasing positive regularizations and maximal completion
gives the same operator comparisons for nonnegative profiles,
including their values at jumps and domain endpoints.
\end{proof}

The geometric Hardy transform also satisfies, whenever the transforms
are defined,
\begin{equation}\label{eq:hardy-geometric-product}
 \mathcal G(q_1q_2)=\mathcal G(q_1)\mathcal G(q_2).
\end{equation}
For $q$ as in Definition~\ref{def:geometrichardytransform}, $c>0$,
and $J_c=(0,\min\{1,1/c\})$, dilation satisfies
\begin{equation}\label{eq:hardy-geometric-dilation}
 \mathcal G\!\left(\left.(q\circ(c\id))\right|_{J_c}\right)
 =\left.(\mathcal G(q)\circ(c\id))\right|_{J_c}.
\end{equation}
Here the transform on $J_c$ uses the same averages over $(0,u)$
as in \eqref{eq:local-geometric-definition}, and all restrictions
are maximally completed.

\begin{lemma}[P\'olya--Knopp inequality \cite{HardyLittlewoodPolya1952}]
\label{lem:local-polya-knopp}
Let $q$ be as in Definition~\ref{def:geometrichardytransform}.
For every $r>0$ and $0<a\leq1$,
\begin{equation}\label{eq:local-polya-knopp}
 \int_0^a\mathcal G(q)^r
 \leq\mathrm e\int_0^a q^r.
\end{equation}
\end{lemma}
For completeness, a proof is given in Section~\ref{app:polya-knopp}.

\subsection{A two-sided extension estimate.}
Localization to sets on which a restriction is Lipschitz is used, for
example, in \cite[Theorem~6.1]{BobkovNayarTetali2017}; the proof of their
Theorem~1.1 also uses a scalar Lipschitz extension.
The following estimate uses both extremal extensions and averages the
cutoff level to obtain the term $a+2m/s$ in
\eqref{eq:local-endpoint-tail}.

\begin{lemma}[Logarithmic endpoint control]\label{lem:local-endpoint}
Let $q_\Delta\in\Md$ satisfy
$(0,1)\subseteq\operatorname{int}\dom q_\Delta$ and have
nonnegative values on $(0,1)$.
Let $Z$ take values in a metric space $(E,d_E)$, let $Z'$ be an
independent copy, and suppose that, for every real $1$-Lipschitz
function $f$ on $E$, $\TT[|f(Z)-f(Z')|]\leq q_\Delta$.
Let $h:E\to\R$ and $U:E\to[0,\infty)$ be measurable, with
$m:=\E\left[U(Z)\right]<\infty$. If, for some $C>0$,
\begin{equation}\label{eq:local-endpoint-modulus}
 \forall x,y\in E:\qquad
 |h(x)-h(y)|\leq C e^{U(x)+U(y)}d_E(x,y),
\end{equation}
then, for $0<a<1$, $s>0$ and $t\in q_\Delta(a)$,
\begin{equation}\label{eq:local-endpoint-tail}
 \PP\bigl(|h(Z)-h(Z')|>Ce^s t\bigr)
 \leq a+\frac{2m}{s}.
\end{equation}
\end{lemma}
\begin{proof}
Put $D=\{|h(Z)-h(Z')|>Ce^s t\}$.
For $0\leq r\leq s/2$, set
$A_r=\{U\leq r\}$ and $B_r=\{U\leq s-r\}$, note that $A_r\subset B_r$.
The restriction of $h$ to $A_r$ is $Ce^s$-Lipschitz. If $A_r$ is
nonempty, its two extremal extensions \cite{mcshane1934extension} are
\[
 h_-(x)=\sup_{z\in A_r}\{h(z)-Ce^s d_E(x,z)\},\qquad
 h_+(x)=\inf_{z\in A_r}\{h(z)+Ce^s d_E(x,z)\}.
\]
These are finite $Ce^s$-Lipschitz functions, equal to $h$ on $A_r$.
For $x\in B_r$ and $z\in A_r$, \eqref{eq:local-endpoint-modulus}
gives $|h(x)-h(z)|\leq Ce^s d_E(x,z)$; hence
$h_-(x)\leq h(x)\leq h_+(x)$ on $B_r$.
For $x,y\in A_r$ and either $\epsilon\in\{-,+\}$, we have
$h_\epsilon(x)-h_\epsilon(y)=h(x)-h(y)$.
For $x\in A_r$ and $y\in B_r\setminus A_r$,
the preceding bounds give
\[
 \begin{aligned}
 h(y)-h(x)>Ce^s t
 &\implies h_+(y)-h_+(x)\geq h(y)-h(x)>Ce^s t,\\
 h(x)-h(y)>Ce^s t
 &\implies h_-(x)-h_-(y)\geq h(x)-h(y)>Ce^s t.
 \end{aligned}
\]
Since $A_r\subseteq B_r$, these implications and the corresponding
ones with $x,y$ exchanged yield the pointwise inequality
\[
 \begin{aligned}
 &\mathbf1_D\bigl(
   \mathbf1_{\{Z\in A_r,\ Z'\in B_r\}}
  +\mathbf1_{\{Z'\in A_r,\ Z\in B_r\}}\bigr)
  \leq\sum_{\epsilon\in\{-,+\}}
   \mathbf1_{\{|h_\epsilon(Z)-h_\epsilon(Z')|>Ce^s t\}}.
 \end{aligned}
\]
Taking expectations and using the symmetry of $Z,Z'$, then applying
Lemma~\ref{lem:survivalcomparison} to $h_\epsilon/(Ce^s)$, gives
\[
 \begin{aligned}
 2\PP(D,\ Z\in A_r,\ Z'\in B_r)
 &\leq\sum_{\epsilon\in\{-,+\}}
       \PP(|h_\epsilon(Z)-h_\epsilon(Z')|>Ce^s t)
 \leq2a.
 \end{aligned}
\]
If $A_r$ is empty, the restricted probability is zero. Exchanging
the two copies gives, for every $0\leq r\leq s$,
\[
 \PP(D,\ U(Z)\leq r,\ U(Z')\leq s-r)\leq a.
\]
Integrate this inequality in $r$. Tonelli's theorem yields
\[
 \E\left[\mathbf1_D(s-U(Z)-U(Z'))_+\right]\leq a s.
\]
Since $(s-v)_+\geq s-v$, it follows that
\[
 s\PP(D)
 \leq a s+\E[\mathbf1_D(U(Z)+U(Z'))]
 \leq a s+2m.
\]
This proves\footnote{The one-set alternative takes $A=\{U\leq s/2\}$,
uses Markov's inequality, and discards pairs outside $A\times A$.
It gives $a+4m/s$, not $a+2m/s$. Thus it does not prove
the constant $3$ below; averaging the two extremal extensions is
the improvement.} \eqref{eq:local-endpoint-tail}.
\end{proof}

\begin{theorem}[Randomized Lipschitz control]
\label{the:concentration_bounded_variations}
Let $(E,d_E)$ and $(E',d_{E'})$ be metric spaces, let $Z$ be an
$E$-valued random variable, and let $Z'$ be an independent copy.
Let $n\geq1$, and let $\Phi:E\to E'$ and, for $i\in[n]$,
$\Lambda_i:E\to[0,\infty)$ be measurable. Suppose that
\begin{equation}\label{eq:local-modulus}
 \forall z,z'\in E:\qquad
 d_{E'}(\Phi(z),\Phi(z'))
 \leq d_E(z,z')\prod_{i=1}^n
             \max\{\Lambda_i(z),\Lambda_i(z')\}.
\end{equation}
Let $q_\Delta,q_1,\ldots,q_n\in\Md$ have domain interior $I$ and
nonnegative values there, with
\begin{equation}\label{eq:local-input}
 \begin{aligned}
 &\forall f:E\to\R\text{ $1$-Lipschitz}:\qquad
   \TT[|f(Z)-f(Z')|]\leq q_\Delta,\\
 &\forall i\in[n]:\qquad \TT[\Lambda_i(Z)]\leq q_i.
 \end{aligned}
\end{equation}
Assume
\[
 \int_0^1\log_+\left(q_\Delta(s)\prod_iq_i(s)\right)\,ds<\infty.
\]
For every $1$-Lipschitz $g:E'\to\R$, if no input profile vanishes
identically on $I$, then
\begin{equation}\label{eq:local-geometric}
 \TT[|g(\Phi(Z))-g(\Phi(Z'))|]
 \leq\left.
 \left[\mathcal G\left(q_\Delta\prod_iq_i\right)
       \circ(\id/3)\right]\right|_I.
\end{equation}
If an input profile vanishes identically on $I$, then
$|g(\Phi(Z))-g(\Phi(Z'))|=0$ almost surely.
In all cases, for every $r>0$ such that
$q_\Delta\prod_iq_i\in L^r(I)$,
\begin{equation}\label{eq:local-moment}
 \bigl\|g(\Phi(Z))-g(\Phi(Z'))\bigr\|_{L^r}
 \leq(3\mathrm e)^{1/r}
       \left\|q_\Delta\prod_iq_i\right\|_{L^r(I)}.
\end{equation}
\end{theorem}

The proof also gives the stronger estimate \eqref{eq:local-optimized},
which allows unequal probability levels.

The following proposition, proved in Section~\ref{app:local-moment},
complements \eqref{eq:local-moment} with a lower bound for the same
prescribed profiles.

\begin{proposition}[Lower bound for the moment scale]
\label{prop:local-moment-sharpness}
Let $q_\Delta,q_1,\ldots,q_n$ satisfy the hypotheses of
Theorem~\ref{the:concentration_bounded_variations}.
There exist a metric space $E$, an $E$-valued random vector $Z$,
measurable controls $\Lambda_i$, and a scalar map $\Phi:E\to\R$
satisfying \eqref{eq:local-modulus} and \eqref{eq:local-input}
with these profiles, such that for every $r>0$,
\begin{equation}\label{eq:local-moment-lower}
 \bigl\|\Phi(Z)-\Phi(Z')\bigr\|_{L^r}
 \geq2^{-1/r}
       \left\|q_\Delta\prod_iq_i\right\|_{L^r(I)}.
\end{equation}
\end{proposition}
The coefficient $2^{-1/r}$ in this universal lower guarantee is
optimal for $r\geq1$. Indeed, take
$q_\Delta=q_1=\cdots=q_n=1$ on $I$ and consider any admissible
scalar map $\Phi$. On the common full-measure set
$A=\bigcap_i\{\Lambda_i\leq1\}$, the modulus
\eqref{eq:local-modulus} makes $\Phi$ $1$-Lipschitz.
A McShane extension from $A$ to $E$, followed by
\eqref{eq:local-input}, gives $|V-V'|\leq1$ almost surely for
$V=\Phi(Z)$. Independence then forces the essential range of $V$
to have diameter at most $1$. After translation, $V\in[0,1]$
almost surely, and its distribution function $F_V$ satisfies
\[
 \forall r\geq1:\qquad
 \E\left[|V-V'|^r\right]\leq\E\left[|V-V'|\right]
 =2\int_0^1F_V(t)(1-F_V(t))\,dt\leq\frac12.
\]
The equiprobable Bernoulli law on $\{0,1\}$, realized by the
construction in Section~\ref{app:local-moment} for these constant
profiles, attains $1/2$.
Thus no larger coefficient can hold in the universal lower guarantee.

\begin{proof}[Proof of Theorem~\ref{the:concentration_bounded_variations}]
Write $h=g\circ\Phi$. Initially suppose that all profiles are
strictly positive on $I$. Fix $0<b<1$, choose
$c_i\in q_i(b)$, and set
\[
 X_i=\log_+\frac{\Lambda_i(Z)}{c_i}.
\]
Naturally, $\TT[X_i]\leq\log_+(q_i/c_i)$ with $\log_+(q_i/c_i)$ that vanishes on $(b,1)$.
Near zero all the other factors of $q_\Delta\prod_jq_j$
have a positive lower bound, so
$\log_+(q_i/c_i)\leq\log_+(q_\Delta\prod_jq_j)+C_i$ there;
away from zero it is bounded. One can then deduce that $\log_+(q_i/c_i)$ is integrable and~\eqref{eq:expectationarea} gives
$X_i\in L^1$.

By the marginal tail bounds, the quantile representation of
expectation \eqref{eq:expectationarea}, and linearity,
\begin{equation}\label{eq:local-logarithmic-mean}
 \begin{aligned}
 \E\left[\sum_iX_i\right]
 &=\sum_i\E X_i
 \leq\sum_i\int_0^1\log_+\frac{q_i(s)}{c_i}\,ds
 =\int_0^b\log\frac{\prod_iq_i(s)}{\prod_i c_i}\,ds
=:I_b.
 \end{aligned}
\end{equation}
The last equality uses $c_i\in q_i(b)$ and monotonicity:
for $s<b$, $q_i(s)\geq c_i$, while for $s>b$, $q_i(s)\leq c_i$.
Thus the positive-part logarithms are ordinary logarithms on $(0,b)$
and vanish on $(b,1)$; their values at $b$ do not affect the integrals.
Moreover, \eqref{eq:local-modulus} and $a\leq e^{\log_+a}$ give
\[
 \begin{aligned}
 |h(z)-h(z')|
 &\leq\left(\prod_i c_i\right)
 \exp\left(\sum_i\log_+
       \frac{\max\{\Lambda_i(z),\Lambda_i(z')\}}{c_i}\right)d_E(z,z')\\
 &\leq\left(\prod_i c_i\right)
 \exp\left(\sum_i\log_+\frac{\Lambda_i(z)}{c_i}
          +\sum_i\log_+\frac{\Lambda_i(z')}{c_i}\right)d_E(z,z').
 \end{aligned}
\]
For $a,b\geq0$, the second inequality uses
$\log_+\max\{a,b\}\leq\log_+a+\log_+b$.
It separates the two endpoints in the form required by
Lemma~\ref{lem:local-endpoint}, with
$U(z)=\sum_i\log_+(\Lambda_i(z)/c_i)$.
The sum in this exponential, evaluated at the two independent
copies, has mean $2\E\left[\sum_iX_i\right]$. If $I_b>0$, apply
Lemma~\ref{lem:local-endpoint} with $C=\prod_i c_i$ and
$s=I_b/b$. Since
\[
 \left(\prod_i c_i\right)e^{I_b/b}
 =\mathcal G\left(\prod_iq_i\right)(b),
\]
this gives for all $a,b\in (0,1)$:
\begin{equation}\label{eq:local-two-parameter}
 \forall t\in q_\Delta(a):\qquad \PP\left(|g(\Phi(Z))-g(\Phi(Z'))|
          >t\mathcal G\left(\prod_iq_i\right)(b)\right)
 \leq a+2b.
\end{equation}
If $I_b=0$, then $\sum_iX_i=0$ almost surely. On this
full-measure set $h$ is Lipschitz with constant $\prod_i c_i$, so
extension gives
\eqref{eq:local-two-parameter} with the smaller right side $a$.
Let $Q(u)=\min\TT[|h(Z)-h(Z')|](u)$, the right-continuous
scalar representative of the output tail quantile. For $0<u<1$,
\eqref{eq:local-two-parameter} gives the stronger bound
\begin{equation}\label{eq:local-optimized}
 Q(u)\leq\inf_{\substack{a,b>0\\a+2b\leq u}}
 q_\Delta(a)\mathcal G\left(\prod_iq_i\right)(b).
\end{equation}
The right side uses scalar selections of the input profiles;
the lower endpoint $Q(u)$ allows equality in the probability budget.
Taking $a=b=u/3$ and maximally completing this scalar bound gives
\begin{equation}\label{eq:local-sharper}
 \TT[|g(\Phi(Z))-g(\Phi(Z'))|]
 \leq\left.
 \left[\left(q_\Delta\mathcal G\left(\prod_iq_i\right)\right)
       \circ(\id/3)\right]\right|_I.
\end{equation}
Lemma~\ref{lem:hardy-geometric-comparison} and the product identity
\eqref{eq:hardy-geometric-product} give
$q_\Delta\mathcal G(\prod_iq_i)
 \leq\mathcal G(q_\Delta\prod_iq_i)$, proving
\eqref{eq:local-geometric}.

If none of the profiles vanishes identically, all have a positive
lower bound on a common interval near zero. The integrability of
$\log_+(q_\Delta\prod_iq_i)$ then implies that of each
$\log_+q_j$, including
$j=\Delta$. Apply the positive case to $q_j+\varepsilon$ and let
$\varepsilon\downarrow0$. For $0<\varepsilon\leq1$,
$\log_+(q_j+\varepsilon)\leq\log2+\log_+q_j$, so the positive
logarithms have an integrable majorant; the negative logarithms decrease,
possibly to $-\infty$. Thus the geometric averages have the stated
limits. Pass to the limit in scalar tail inequalities at common
continuity levels. At jump levels in \eqref{eq:local-optimized},
approach $a,b,u$ from above while keeping $a+2b\leq u$.
Right continuity of $Q$ and the right limits of the input profiles,
which are no larger than any selected values, give the same bound.
Proposition~\ref{prop:comparison} and maximal completion give
the operator bounds, including endpoints.
If some $q_i$ vanishes identically, then $\Lambda_i(Z)=0$ almost
surely and \eqref{eq:local-modulus} gives
$|g(\Phi(Z))-g(\Phi(Z'))|=0$. If instead
$q_\Delta$ vanishes identically, restrict $h$ to
$A_M=\{\max_i\Lambda_i\leq M\}$. For all sufficiently large $M$,
this set is nonempty and $h|_{A_M}$ has a Lipschitz extension whose
copy difference is zero almost surely. Hence
\[
 \PP\bigl(|g(\Phi(Z))-g(\Phi(Z'))|>0\bigr)
 \leq2\PP(Z\notin A_M)\to0.
\]

The moment bound is immediate in these zero-profile cases.
Otherwise, the quantile representation,
\eqref{eq:local-geometric} and the P\'olya--Knopp inequality
\eqref{eq:local-polya-knopp} from
Lemma~\ref{lem:local-polya-knopp} give
\[
 \begin{aligned}
 \E\left[|g(\Phi(Z))-g(\Phi(Z'))|^r\right]
 &\leq\int_0^1
 \mathcal G\left(q_\Delta\prod_iq_i\right)(u/3)^r\,du\\
 &=3\int_0^{1/3}
 \mathcal G\left(q_\Delta\prod_iq_i\right)(v)^r\,dv\\
 &\leq3\mathrm e\int_0^1
 \left(q_\Delta(v)\prod_iq_i(v)\right)^r\,dv.
 \end{aligned}
\]
\end{proof}
\begin{remark}[Centering and unequal probability levels]
For $V=g(\Phi(Z))$ and any median $m_V$, the independent-copy
comparison \eqref{eq:independentcopy} gives, when no input profile
vanishes identically,
\begin{equation}\label{eq:local-median}
 \TT[|V-m_V|]\leq
 \left.
 \left[\left(q_\Delta\mathcal G\left(\prod_iq_i\right)\right)
       \circ(\id/6)\right]\right|_I.
\end{equation}
In all cases, for every $r>0$ such that
$q_\Delta\prod_iq_i\in L^r(I)$,
\[
 \|V-m_V\|_{L^r}\leq(6\mathrm e)^{1/r}
              \left\|q_\Delta\prod_iq_i\right\|_{L^r(I)}.
\]
The zero-profile cases are interpreted as above.
For $r\geq1$ and $q_\Delta\prod_iq_i\in L^r(I)$, integrability follows and
conditional Jensen gives
\begin{equation}\label{eq:local-mean}
 \|V-\E V\|_{L^r}\leq\|V-V'\|_{L^r}
 \leq(3\mathrm e)^{1/r}
       \left\|q_\Delta\prod_iq_i\right\|_{L^r(I)}.
\end{equation}
\end{remark}

\subsection{Application to random feed-forward networks}
\label{subsec:local-network}

Let $L,N,p_0,\ldots,p_L$ be positive integers, let
$X\in\R^{p_0\times N}$ and, for $1\leq\ell\leq L$,
$W_\ell\in\R^{p_\ell\times p_{\ell-1}}$
be random matrices, and let the entrywise activations be
$1$-Lipschitz with $\sigma_\ell(0)=0$.
Write $Z=(X,W_1,\ldots,W_L)$ and define
\begin{equation}\label{eq:local-network}
 H_0=X,\qquad
 H_\ell=\sigma_\ell\!\left(
 \frac{W_\ell H_{\ell-1}}{\sqrt{p_\ell+p_{\ell-1}}}\right),
 \qquad
 \Phi(Z)=\frac{H_L}{\sqrt{p_L}}.
\end{equation}
Equip the parameter space with the Euclidean product of
Hilbert--Schmidt norms. Let $m_x,m_\ell$ be medians of
$\|X\|_{\mathrm{HS}}/\sqrt{Np_0}$ and
$\|W_\ell\|_{\mathrm{op}}/\sqrt{p_\ell+p_{\ell-1}}$,
respectively, and set
\[
 D=\frac1{\sqrt{p_L}}
 \left(1+Np_0\sum_{\ell=1}^L
             \frac1{p_\ell+p_{\ell-1}}\right)^{1/2}.
\]

\begin{proposition}[Network tail profiles]\label{prop:local-network}
Suppose that $Z$ satisfies the first inequality in
\eqref{eq:local-input} with a nonnegative profile $q_\Delta$ on
$I=(0,1)$ such that $\int_0^1\log_+q_\Delta<\infty$.
For every Hilbert--Schmidt $1$-Lipschitz observation $g$, put
$V=g(\Phi(Z))$ and let $m_V$ be a median. Then, for $0<u<1$,
\begin{equation}\label{eq:local-network-profile}
 \begin{aligned}
 \TT[|V-m_V|](u)
 &\leq Dq_\Delta(u/6)\\
 &\quad{}\times
 \prod_{\ell=1}^L
 \mathcal G\!\left(
 \max\{m_x,m_\ell\}
 +\frac{q_\Delta}{\sqrt{\min\{Np_0,p_\ell+p_{\ell-1}\}}}
 \right)(u/12).
 \end{aligned}
\end{equation}
Profiles use the completion convention on $I$.
\end{proposition}

\begin{proof}
For $z=(x,w_1,\ldots,w_L)$, put
\[
 \Lambda_\ell(z)=\max\left\{
 \frac{\|x\|_{\mathrm{HS}}}{\sqrt{Np_0}},
 \frac{\|w_\ell\|_{\mathrm{op}}}{\sqrt{p_\ell+p_{\ell-1}}}
 \right\}.
\]
Telescoping the layer differences and applying Cauchy--Schwarz gives
\begin{equation}\label{eq:local-network-two-point}
 \|\Phi(z)-\Phi(z')\|_{\mathrm{HS}}
 \leq D\|z-z'\|_2
       \prod_{\ell=1}^L\max\{\Lambda_\ell(z),\Lambda_\ell(z')\}.
\end{equation}
Indeed, after extracting the product, the input difference contributes
$\|x-x'\|_{\mathrm{HS}}/\sqrt{p_L}$ and the $\ell$th weight difference
contributes
\[
 \sqrt{Np_0/p_L}\,
 \|w_\ell-w'_\ell\|_{\mathrm{HS}}/\sqrt{p_\ell+p_{\ell-1}}.
\]

For a scalar $U$, a median $m_U$, and an independent copy $U'$, the
two orientations of $\{U>m_U+t,U'\leq m_U\}$ give
$\PP(U>m_U+t)\leq\PP(|U-U'|>t)$.
Apply this to the two normalized norms and take a union bound:
\[
 \TT[\Lambda_\ell(Z)]\leq q_\ell,\qquad
 q_\ell(v)=\max\{m_x,m_\ell\}
 +\frac{q_\Delta(v/2)}{\sqrt{\min\{Np_0,p_\ell+p_{\ell-1}\}}}.
\]
The logarithmic integrability assumption ensures that these profiles
satisfy the hypotheses of
Theorem~\ref{the:concentration_bounded_variations}.
Apply that theorem to $(g\circ\Phi)/D$, followed by
\eqref{eq:local-median}, the dilation identity
\eqref{eq:hardy-geometric-dilation}, and
\eqref{eq:hardy-geometric-product}.
\end{proof}

\begin{example}[Subgaussian concentration and depth]
Suppose $q_\Delta(u)\leq A\sqrt{\log(2/u)}$ and
$p_0=\cdots=p_L=d$. Set $R(u)=\sqrt{\log(24\mathrm e/u)}$.
By \eqref{eq:hardy-geometric-comparison} and Jensen's inequality,
\begin{equation}\label{eq:local-network-depth}
 \PP\!\left(
 \begin{aligned}
 |V-m_V|&>
 A\sqrt{\frac{1+NL/2}{d}}\,R(u)
 \\
 &\quad{}\times\left(
 \max\{m_x,m_1,\ldots,m_L\}
 +\frac{AR(u)}{\sqrt{d\min\{N,2\}}}
 \right)^L
 \end{aligned}
 \right)\leq u.
\end{equation}
The confidence constants do not grow with $L$, but the complete
bound does. Independent standard Gaussian entries allow $A=2$ and
medians satisfying $m_x\leq1+\sqrt{2\log2/(Nd)}$ and
$m_\ell\leq\sqrt2+\sqrt{\log2/d}$
\cite[Theorem~5.32 and Proposition~5.34]{Vershynin2012Introduction}.

For fixed $L$, $N\asymp d$, and uniformly bounded $A,m_x,m_\ell$,
\eqref{eq:local-network-depth} gives, up to constants depending on
these bounds and on $L$, the tail rate
$\min\{t^2,d^{L/(L+1)}t^{2/(L+1)}\}$ of Adamczak
\cite[Theorem~3.7]{Adamczak2025NonlinearMatrices}.
In this subgaussian setting, the concentration behavior is
essentially the same as in Adamczak's result; his work covers many
more settings, including batch estimates
\cite[Theorem~3.14]{Adamczak2025NonlinearMatrices}.
\end{example}

Our contribution beyond this subgaussian setting is to cover
non-subgaussian joint laws of the data and weights, as illustrated
by the following power-tail example.

\begin{example}[Power tails]
If, for some $A,\alpha>0$, $q_\Delta(u)\leq Au^{-\alpha}$, then
\begin{equation}\label{eq:local-network-power}
 \begin{aligned}
 \TT[|V-m_V|](u)
 &\leq DAe^{\alpha L}(6/u)^\alpha\\
 &\quad{}\times
 \prod_{\ell=1}^L\left(
 \max\{m_x,m_\ell\}
 +\frac{A(12/u)^\alpha}
        {\sqrt{\min\{Np_0,p_\ell+p_{\ell-1}\}}}
 \right).
 \end{aligned}
\end{equation}
This follows by using, for $m,c\geq0$,
$\mathcal G(m+c\,\id^{-\alpha})(v)
 \leq e^\alpha(m+cv^{-\alpha})$.
The far-tail quantile power is $\alpha(L+1)$.
The hypothesis concerns the joint law, not merely marginal
entrywise tails.
\end{example}

\appendix
\addtocontents{toc}{\protect\setcounter{tocdepth}{1}}
\renewcommand{\thesubsection}{\Alph{section}.\arabic{subsection}}
\renewcommand{\thetheorem}{\Alph{section}.\arabic{theorem}}
\section{Additional proofs}
\label{app:proofs}

\subsection{Proof of the Hardy profile bounds}
\label{app:hardyprofiles}

\begin{proof}[Proof of Proposition~\ref{prop:hardyprofiles}]
On $(0,L)$, write $q=\alpha^{-1}$. In the power case, the increasing convex function $t\mapsto\alpha(t)^{-1/\kappa}$ has a concave inverse defined for $y>L^{-1/\kappa}$ by $g(y)=q(y^{-\kappa})$. Fix $0<p<L$. For $0<r<1$, monotonicity gives $q(pr)\geq q(p)$, bounding its negative part. A supporting affine upper bound for $g$ bounds the positive part by a constant times $1+r^{-1/\kappa}$, integrable because $\kappa>1$. Thus $\mathcal H(q)(p)$ is well defined before applying Jensen's inequality. Using \eqref{eq:hardyprofileidentities},
\[
 \mathcal H(q)(p)=\int_0^1g\bigl((pr)^{-1/\kappa}\bigr)\,dr
 \leq g\left(\frac\kappa{\kappa-1}p^{-1/\kappa}\right)
 =q\left(\left(\frac{\kappa-1}\kappa\right)^\kappa p\right).
\]
With $c=((\kappa-1)/\kappa)^\kappa$, separated comparison and completion give $\mathcal H(q)\leq q\circ(c\id)$, whose inverse is $\alpha/c$. Proposition~\ref{prop:inversion} gives \eqref{eq:hardypowerbound}.

In the exponential case, the increasing convex function $-\log\alpha$ has a concave inverse defined for $y>-\log L$ by $g(y)=q(e^{-y})$. Its supporting affine upper bound and the lower bound $q(pr)\geq q(p)$ ensure absolute integrability, since $\int_0^1|\log r|\,dr<\infty$. Jensen's inequality gives
\[
 \mathcal H(q)(p)=\int_0^1g(-\log(pr))\,dr
 \leq g(1-\log p)=q(p/e).
\]
The same comparison and inversion prove \eqref{eq:hardyexponentialbound}. When $L=\infty$, the lower endpoints of the two domains of $g$ are $0$ and $-\infty$, respectively. The equality cases follow from \eqref{eq:hardyregularizedprofiles}.
\end{proof}

For a positive scalar interior profile $\alpha$, convexity of $-\log\alpha$ implies that, for every $\kappa>0$, $\alpha^{-1/\kappa}$ is convex, by composition with an increasing exponential. Conversely, if these power transforms are convex for exponents tending to infinity, the convex functions $\kappa(\alpha^{-1/\kappa}-1)$ converge pointwise to $-\log\alpha$. Also $(\kappa/(\kappa-1))^\kappa\to e$, matching the constants in \eqref{eq:hardypowerbound} and \eqref{eq:hardyexponentialbound}.

\subsection{Proof of the convex transformation criterion}
\label{app:convextransforms}

For the proof of Proposition~\ref{prop:convextransforms}, we use the
classical weighted-mean limit below
\cite{Qi2000NewProofsWeightedPowerMean}.
\begin{lemma}[Power--geometric limit]\label{lem:powergeometric}
For $a,b>0$, $\kappa>0$ and $0\leq\lambda\leq1$, set
\[
 m_{\kappa,\lambda}(a,b)
 =\bigl(\lambda a^{-1/\kappa}+(1-\lambda)b^{-1/\kappa}\bigr)^{-\kappa}.
\]
Then $m_{\kappa,\lambda}(a,b)\to a^\lambda b^{1-\lambda}$ as $\kappa\to\infty$.
\end{lemma}
\begin{proof}
The expression inside parentheses equals
$1-\kappa^{-1}(\lambda\log a+(1-\lambda)\log b)+o(\kappa^{-1})$.
Taking its logarithm and multiplying by $-\kappa$ proves the limit.
\end{proof}

\begin{proof}[Proof of Proposition~\ref{prop:convextransforms}]
Power convexity gives, for $a,b>0$ and $0\leq\lambda\leq1$,
\[
 \varphi\bigl(m_{\kappa,\lambda}(a,b)\bigr)
 \leq\lambda\varphi(a)+(1-\lambda)\varphi(b).
\]
Let $\kappa\to\infty$ in Lemma~\ref{lem:powergeometric}, then set $a=e^{-x}$ and $b=e^{-y}$ to obtain exponential convexity.
Conversely, the arithmetic--geometric inequality and monotonicity of $\varphi$ give
\[
 \varphi\bigl((\lambda a+(1-\lambda)b)^{-\kappa}\bigr)
 \leq\varphi\bigl(a^{-\kappa\lambda}b^{-\kappa(1-\lambda)}\bigr)
 \leq\lambda\varphi(a^{-\kappa})+(1-\lambda)\varphi(b^{-\kappa}),
\]
where the second inequality is exponential convexity. Apply the argument to $-\varphi$ for concavity.
\end{proof}

\subsection{Proof of the moment lower bound}
\label{app:local-moment}

\begin{proof}[Proof of Proposition~\ref{prop:local-moment-sharpness}]
In this proof, identify the profiles with nonnegative scalar
selections on $(0,1)$. Let $U$ be uniform on $(0,1)$ and let $B$
be an independent Bernoulli variable of parameter $1/2$. Define
\[
 Z=
 \begin{cases}
 (q_\Delta(U),q_1(U),\ldots,q_n(U)),
   &\begin{gathered}
      B=1\text{ and all these}\\
      \text{coordinates are positive},
    \end{gathered}\\
 0,&\text{otherwise}.
 \end{cases}
\]
Let $E$ consist of the origin and, for $0<u<1$, the positive vectors
$(q_\Delta(u),q_1(u),\ldots,q_n(u))$.
For $z=(z_0,\ldots,z_n)\in E$, set
\[
 \Lambda_i(z)=z_i,\qquad
 \Phi(z)=\prod_{i=0}^n z_i,
 \qquad
 d_E(z,z')=
 \begin{cases}
 0,&z=z',\\[2mm]
 \displaystyle
 \frac{|\Phi(z)-\Phi(z')|}
      {\prod_{i=1}^n\max\{z_i,z_i'\}},&z\ne z'.
 \end{cases}
\]
The denominator is positive whenever $z\ne z'$. The points of $E$
are totally ordered coordinatewise, and the product is strictly
increasing between distinct positive points. To check the triangle
inequality, write $P(w)=\prod_{i=1}^nw_i$.
For $x\leq y\leq z$ with $y\ne0$, direct substitution gives
\[
 \begin{aligned}
 d_E(x,z)-d_E(x,y)
 &=z_0-y_0+x_0P(x)\left(\frac1{P(y)}-\frac1{P(z)}\right)\geq0,\\
 d_E(x,z)-d_E(y,z)
 &=\frac{P(y)}{P(z)}d_E(x,y)\geq0,\\
 d_E(x,y)+d_E(y,z)-d_E(x,z)
 &=\left(1-\prod_{i=1}^n\frac{y_i}{z_i}\right)d_E(x,y)\geq0.
 \end{aligned}
\]
These formulas also allow $x=0$. If $y=0$, then $x=0$ and the
triangle inequalities are immediate. Thus
\[
 \max\{d_E(x,y),d_E(y,z)\}
 \leq d_E(x,z)\leq d_E(x,y)+d_E(y,z).
\]
Hence $d_E$ is a metric. For $1\leq i\leq n$, take the coordinate
controls $\Lambda_i(z)=z_i$. Then \eqref{eq:local-modulus}
holds with equality. Its topology agrees with the usual relative topology
away from the origin, so its Borel sets are the usual relative
Borel sets on $E$; all the preceding maps are therefore measurable.

For every $z,z'\in E$,
\[
 d_E(z,0)=z_0,\qquad
 d_E(z,z')\leq\max\{z_0,z_0'\}.
\]
Consequently, for every $1$-Lipschitz $f:E\to\R$ and $t\geq0$,
\[
 \PP(|f(Z)-f(Z')|>t)
 \leq2\PP(Z_0>t)
 \leq\PP(q_\Delta(U)>t).
\]
Likewise, for $t\geq0$, the event $\{\Lambda_i(Z)>t\}$ implies
$B=1$ and $q_i(U)>t$, since $\Lambda_i(Z)=Z_i$ and $Z=0$ whenever
$B=0$. Independence of $B$ and $U$ therefore gives
\[
 \PP(\Lambda_i(Z)>t)
 \leq\PP(B=1,\ q_i(U)>t)
 =\frac12\PP(q_i(U)>t)
 \leq\PP(q_i(U)>t).
\]
The quantile representation in Example~\ref{ex:survival} and
Lemma~\ref{lem:survivalcomparison} give \eqref{eq:local-input}.

Finally, $\Phi\geq0$ on $E$, $\Phi(0)=0$, and $\PP(Z=0)\geq1/2$.
On $\{Z'=0\}$ the integrand $|\Phi(Z)-\Phi(Z')|^r$ equals
$\Phi(Z)^r$, and on $\{Z=0\}$ it equals $\Phi(Z')^r$.
Both contributions vanish on the intersection $\{Z=Z'=0\}$,
so retaining them and using independence gives, for every $r>0$,
\[
 \begin{aligned}
 \E\left[|\Phi(Z)-\Phi(Z')|^r\right]
 &\geq\E[\Phi(Z)^r\mathbf1_{\{Z'=0\}}]
       +\E[\Phi(Z')^r\mathbf1_{\{Z=0\}}]\\
 &=2\PP(Z=0)\,\E[\Phi(Z)^r]\\
 &\geq\E[\Phi(Z)^r]
 =\frac12\int_0^1
   \left(q_\Delta(u)\prod_{i=1}^nq_i(u)\right)^r\,du.
 \end{aligned}
\]
Taking $r$th roots proves \eqref{eq:local-moment-lower} directly for the
prescribed profiles, without approximation.
\end{proof}

\subsection{Proof of the P\'olya--Knopp inequality}
\label{app:polya-knopp}

\begin{proof}[Proof of Lemma~\ref{lem:local-polya-knopp}]
To prove \eqref{eq:local-polya-knopp}, put $W=q^r$. The quantity
$\mathcal G(q)(u)^r$ agrees almost everywhere with the geometric average
$\exp(u^{-1}\int_0^u\log W)$, so it suffices to bound the integral
of this average by $\mathrm e\int_0^a W$.
Assume the latter integral is finite, since otherwise the claim is
immediate, and first replace $W$ by $W+\varepsilon$, $\varepsilon>0$,
to keep logarithms finite. Writing $W$ for this positive function,
apply the geometric--arithmetic mean inequality to $sW(s)$ on
$(0,u)$. Since
\[
 \frac1u\int_0^u\log s\,ds=\log u-1,
\]
we obtain
\[
 \exp\left(\frac1u\int_0^u\log W(s)\,ds\right)
 =\frac{\mathrm e}{u}
   \exp\left(\frac1u\int_0^u\log(sW(s))\,ds\right)
 \leq\frac{\mathrm e}{u^2}\int_0^u sW(s)\,ds.
\]
Tonelli's theorem exchanges the nonnegative integrals over
$0<s<u<a$, giving
\[
 \begin{aligned}
 \int_0^a\frac{\mathrm e}{u^2}\int_0^u sW(s)\,ds\,du
 &=\mathrm e\int_0^a sW(s)
                  \left(\int_s^a\frac{du}{u^2}\right)\,ds\\
 &=\mathrm e\int_0^a\left(1-\frac{s}{a}\right)W(s)\,ds
 \leq\mathrm e\int_0^a W(s)\,ds.
 \end{aligned}
\]
This proves the desired bound for the positive regularizations.
Letting $\varepsilon\downarrow0$ includes zeros, with
$\log0=-\infty$ and $\exp(-\infty)=0$, and proves
\eqref{eq:local-polya-knopp}.
\end{proof}

\subsection{Proof of the odd-median formula}
\label{app:mediantransform}

\begin{proof}[Proof of Proposition~\ref{prop:mediantransform}]
Let $I\subseteq\{1,\ldots,n\}$ satisfy
\[
 |I|=r+\ell.
\]
For $x\in[0,\infty)^n$, at most $r$ coordinates are smaller than
$M(x)$. Hence at
least $\ell$ coordinates indexed by $I$ are not smaller than $M(x)$.
Therefore
\begin{equation}\label{eq:medianlinearupper}
 \ell M(x)\le\sum_{i\in I}x_i.
\end{equation}
For every finite family $v^{(j)}\in[0,\infty)^n$ with
$\sum_jv^{(j)}\le a$,
\[
 \sum_jM(v^{(j)})
 \le
 \frac1\ell\sum_{i\in I}\sum_jv_i^{(j)}
 \le
 \frac1\ell\sum_{i\in I}a_i.
\]
Choose $I$ to contain the $r+\ell$ smallest coordinates of $a$, and
then minimize over $\ell$. This gives
\[
 M^\sharp(a)
 \le
 \min_{1\le\ell\le r+1}
 \frac{a_{(1)}+\cdots+a_{(r+\ell)}}{\ell}.
\]

For the reverse inequality, put
\[
 S_k:=\sum_{i=1}^ka_{(i)},
 \qquad
 S_0:=0,
 \qquad t:=
 \min_{1\le\ell\le r+1}
 \frac{S_{r+\ell}}{\ell}.
\]
Let $k$ be the number of coordinates strictly below $t$. Then
\[
 \sum_{i=1}^n\min(a_i,t)=S_k+(n-k)t.
\]
If $k\leq r$, this is at least $(r+1)t$. If $k>r$, the definition
of $t$ gives $S_k\geq(k-r)t$, with the same conclusion.
Conversely, choose $j=r+\ell_*$ attaining the minimum defining $t$.
Then
\[
 \sum_{i=1}^n\min(a_i,t)
 \leq S_j+(n-j)t=(j-r)t+(n-j)t=(r+1)t.
\]
Hence
\begin{equation}\label{eq:mediantruncation}
 \sum_{i=1}^n\min(a_i,t)=(r+1)t.
\end{equation}

If $t=0$, the reverse inequality follows from $M^\sharp(a)\geq0$.
Suppose henceforth that $t>0$ and set
\[
 b_i:=\min(a_i,t).
\]
Then
\[
 \frac bt
 \in
 \left\{
 x\in[0,1]^n:
 \sum_{i=1}^{n}x_i=r+1
 \right\}.
\]
A point with two coordinates strictly between $0$ and $1$ admits
opposite small perturbations of those coordinates that preserve the sum,
so it is not an extreme point. The integer sum excludes exactly one
fractional coordinate. Conversely, the $0$--$1$ vectors in this polytope
are vertices of the cube and remain extreme points. Thus its extreme
points are precisely $\mathbbm{1}_I$ with $|I|=r+1$. Hence
\[
 b
 =
 \sum_jw_j\mathbbm{1}_{I_j},
 \qquad
 w_j\ge0,
 \qquad
 |I_j|=r+1,
 \qquad
 \sum_jw_j=t.
\]
Since $b\le a$, this family is admissible in
\eqref{eq:superadditivetransform}. Each vector $w_j\mathbbm{1}_{I_j}$ has
median $w_j$, so
\[
 M^\sharp(a)
 \ge
 \sum_jM(w_j\mathbbm{1}_{I_j})
 =
 \sum_jw_j
 =
 t.
\]
This proves \eqref{eq:mediantransform}.
\end{proof}

\section*{Acknowledgements}
I thank Sicheng Tan for interesting discussions on the subadditivity of
expected shortfall. I also thank Patrick Combettes for directing me to
Rockafellar's approach while I was studying concentration tails.

\section*{Funding sources}
This work was supported by a startup fund from The Chinese University of
Hong Kong, Shenzhen, and by the National Natural Science Foundation of China
(General Program, Research Fund for International Scientists (RFIS),
grant no.~W2533014).

\section*{AI use declaration}
\noindent\textbf{Statement:} \emph{During the preparation of this work the author
used ChatGPT to check proofs, improve inequality sharpness, create figures,
check grammar, and produce the complete Lean verification package.
After using this tool, the author reviewed and edited
the content as needed and takes full responsibility for the content of the
published article.}

\section*{Formal verification}
All mathematical results presented in this work have been formalized in
Lean~4 and checked by its kernel. The verification package, including a
paper-to-Lean correspondence and instructions for reproducing the checks,
is documented in the companion GitHub repository:
\url{https://github.com/cosmital/calculus_max_monotone_op_tail_control}.

\bibliographystyle{elsarticle-num}
\bibliography{biblio}
\end{document}